\documentclass[12pt]{article}

\usepackage{amsmath, amssymb, amsfonts, amsbsy, amscd, latexsym, amsthm,color}
\date{}
\newtheorem{Theorem}{Theorem}[section]
\newtheorem{Proposition}[Theorem]{Proposition}

\newtheorem{Lemma}[Theorem]{Lemma}
\theoremstyle{definition}
\newtheorem{Definition}[Theorem]{Definition}

\theoremstyle{remark}
\newtheorem{Remark}[Theorem]{Remark}
 
\numberwithin{equation}{section}

\def\gg#1{\Gamma(#1)}

\def\C{\mathbb{C}}
\def\nr{\nonumber \\}
\def\ser#1#2{{#1}_{#2}(\vec{x})}

\def\sinpi#1{\sin\pi#1}
\def\Bb#1#2#3#4{A(#1,#2,#3,#4)}

\def\ds{\displaystyle}
\title{Relations among complementary and supplementary pairings of
Saalsch\"utzian ${}_4F_3(1)$ series}
\author{R. M. Green \and Ilia D. Mishev \and Eric Stade}

\begin{document}

\maketitle

\begin{abstract}
We investigate  sums $K(\vec{x})$ and $L(\vec{x})$ of pairs of (suitably normalized) Saalsch\"utzian ${}_4F_3(1)$ hypergeometric series, and develop a theory of relations among these $K$ and $L$ functions.

The function $L(\vec{x})$ has been studied extensively in the literature, and has been shown to satisfy a number of two-term and three-term relations with respect to the variable $\vec{x}$.  More recent works have framed these relations in terms of   Coxeter group actions on $\vec{x}$,   and have developed a similar theory of two-term and three-term relations for $K(\vec{x})$.

In this article, we derive ``mixed'' three-term relations, wherein any one of the $L$ (respectively, $K$) functions arising in the above context may be expressed as a linear combination of two of the above $K$ (respectively, $L$) functions.  We show that, under the appropriate Coxeter group action,  the resulting set of three-term relations (mixed and otherwise) among $K$ and $L$ functions partitions  into eighteen orbits. We  provide an explicit example of a relation from each orbit.

We further classify the eighteen orbits into five types, with each type uniquely determined by the distances (under a certain natural metric) between the $K$ and $L$ functions in the relation.  We show that the type of a relation dictates the complexity (in terms of both number of summands and number of factors in each summand) of the coefficients of the $K$ and $L$ functions therein.
\end{abstract}

\section{Introduction}
The theory of hypergeometric series of type ${}_2F_1$ was
systematically developed by Gauss \cite{Ga}. Subsequently,
generalized hypergeometric series of type ${}_pF_q$, where $p$ and
$q$ are positive integers, were studied in the late nineteenth and
early twentieth century by Thomae \cite{T}, Barnes \cite{Bar1,Bar2},
Ramanujan  \cite{Har}, Whipple \cite{Wh1,Wh2,Wh3}, Bailey
\cite{Ba1,Ba}, and others. In particular, relations among ${}_3F_2,
{}_4F_3, {}_7F_6$, and ${}_9F_8$ series of unit argument were obtained.  More recently, a number of these relations, as well as analogous relations among
 basic hypergeometric series, have been put into
group-theoretic frameworks.  See the papers of Beyer et al.\ \cite{BLS}, Srinavasa
Rao et al.\ \cite{Rao}, Formichella et al.\ \cite{FGS}, Roy \cite{Roy},
Mishev \cite{M1}, Van der Jeugt and Srinavasa Rao \cite{V}, and Lievens
and Van der Jeugt \cite{LJ1,LJ2}. Other papers of interest include those of
Groenevelt \cite{Gro}, van de Bult et al.\ \cite{BRS}, and
Krattenthaler and Rivoal \cite{KR}.

Among numerous relatively recent applications, hypergeometric series have played a role in   the theory of automorphic functions, cf. 
Bump \cite{Bu}, Stade \cite{St1,St2,St3,St4}, and Stade and Taggart
\cite{ST}; and in atomic and
molecular physics, cf.   works by Drake \cite{Dra}, Grozin
\cite{Groz}, and Raynal \cite{R}.

In this paper, we expand on the work  in \cite{FGS} and in \cite{M1} concerning  certain sums of pairs of hypergeometric series.

More specifically:  Formichella et al.,  in \cite{FGS}, study a function denoted $K(a;b,c,d;e,f,g)$,
where 
$a,b,c,d,e,f,g \in \mathbb{C}$ and $e+f+g-a-b-c-d=1$, which is
a linear combination of  two Saalsch\"utzian ${}_4F_3(1)$ hypergeometric series. (This function $K(a;b,c,d;e,f,g)$ arises in the theory of archimedean
zeta integrals for automorphic $L$ functions, cf.  \cite{St4,ST}.) It
is shown that   $K(a;b,c,d;e,f,g)$ is invariant under a certain action of the symmetric group $S_6$ on the affine hyperplane \begin{equation*}V=\{(a,b,c,d,e,f,g)^T\in \mathbb{C}^7\colon e+f+g-a-b-c-d=1\},\end{equation*} and thus the
existence of 720 two-term relations for $K(a;b,c,d;e,f,g)$ is
demonstrated. Furthermore, a set of 4960 three-term relations
satisfied by $K(a;b,c,d;e,f,g)$ is described, and is classified into five families, based on the
notion of Hamming type.

Analogously, Mishev investigates, in \cite{M1}, a function $L(a,b,c,d;e;f,g)$ that is a
different linear combination of two Saalsch\"utzian ${}_4F_3(1)$
hypergeometric series.   (The
 series $L(a,b,c,d;e;f,g)$ arises,
in \cite{St2}, in the evaluation of the Mellin transform of a
spherical principal series $GL(4,\mathbb{R})$ Whittaker function, and has been studied in other contexts, as described below.)  It is shown in \cite{M1} that the Coxeter group
$W(D_5)$, which has 1920 elements, is an invariance group for the
function $L(a,b,c,d;e;f,g)$, and thus the existence of 1920 two-term
relations for $L(a,b,c,d;e;f,g)$ is demonstrated. Moreover, it is shown that $L(a,b,c,d;e;f,g)$ satisfies  220 three-term relations, which are classified into
two families based on the notion of $L$-coherence.

As noted above, the $L$ function  has been considered elsewhere.
Versions of this function, written in terms of very-well-poised
${}_7F_6(1)$ series, were studied previously by Bailey \cite{Ba1},
Whipple \cite{Wh3}, and Raynal \cite{R}. (Those authors study two-term
and three-term relations, but do not put their results in a
group-theoretic framework.) The $L$ function also appears as a Wilson
function (a nonpolynomial extension of the Wilson polynomial) in
\cite{Gro}. Van de Bult et al. \cite{BRS} examine generalizations to
elliptic, hyperbolic, and trigonometric hypergeometric functions. A
basic hypergeometric series analog of the $L$ function, in terms of
${}_8\phi_7$ series, was studied by Van der Jeugt and Srinavasa Rao
\cite{V} and by Lievens and Van der Jeugt \cite{LJ1}.

In this paper, we build up on the three-term relations derived in
\cite{FGS} and \cite{M1}. In particular, we obtain three-term
relations that involve one $K$ and two $L$ functions, called
$(K,L,L)$ relations, and three-term relations that involve one $L$
and two $K$ functions, called $(L,K,K)$ relations. We further
classify the $(K,L,L)$ relations into four families and the $(L,K,K)$
relations into seven families. The classification is done according to the
theory of group actions, as follows.

We consider the action of a matrix group $M$,  
isomorphic to the Coxeter group $W(D_6)$ of order 23040, on the hyperplane $V$.  
(The group $M$ is the same, elementwise, as the groups $M_K$ and $M_L$
considered in \cite{FGS} and \cite{M1}, respectively.) We further
consider subgroups $G_K$ and $G_L$ of $M$: $G_K$, which  is  isomorphic to $S_6$,  is an invariance group for $K(a;b,c,d;e,f,g)$, while $G_L$, which is isomorphic to the Coxeter group $W(D_5)$ of order 1920,  is an invariance group for $L(a,b,c,d;e;f,g)$.

For $\vec{x} =(a,b,c,d,e,f,g)^T\in V$, let us denote  $K(a;b,c,d;e,f,g)$ and $L(a,b,c,d;e;f,g)$ by $K(\vec{x})$ and
$L(\vec{x})$, 
respectively.
We show in this work that:

\begin{enumerate}

\item[(a)] For any $  \sigma, \tau, \mu \in M$ 
such that $\tau$ and
$\mu$ are in different right cosets of $G_L$ in $M$, there exists a
three-term relation involving $K(\sigma\vec{x}), L(\tau\vec{x})$,
and $L(\mu\vec{x})$. The total number of such $(K,L,L)$ relations is $32
\times \binom{12}{2} = 2112$.  
These 2112  relations fall into four orbits, of sizes 960,
480, 480, and 192 respectively, under the action of $M$ by right multiplication on 
$(G_K \backslash M) \times S(L^2)$.  Here,
$S(L^2)$ denotes the set of two-element subsets of $G_L \backslash M$.

\item[(b)] For any $\sigma, \tau, \mu \in M$ 
such that $\tau$ and
$\mu$ are in different right cosets of $G_K$ in $M$, there
exists a three-term relation involving $L(\sigma\vec{x}),
K(\tau\vec{x})$, and $K(\mu\vec{x})$. 
The total number of such
$(L,K,K)$ relations is $12 \times \binom{32}{2} = 5952$. 
These 5952 $(L,K,K)$ relations fall into
seven orbits, of sizes 1920, 960, 960, 960, 480, 480, and 192 respectively, under the action of $M$ by right multiplication on 
$(G_L \backslash M) \times S(K^2)$ (where $S(K^2)$ denotes the set of two-element subsets of $G_K \backslash M$).

\end{enumerate}

The above results generalize expressions found in \cite{M} for 
$K(\vec{x})$ as a sum of two particular $L$ functions,  and for 
$L(\vec{x})$ as a sum of two particular $K$ functions.

Thus there are eighteen distinct families of three-term relations for $K$ and $L$ functions:  two relating three $L$ functions, five relating three $K$ functions, four relating a $K$ function to a pair of $L$ functions, and seven relating an $L$ function to a pair of $K$ functions.  In this paper, we present an explicit relation from each family.  We further classify the  eighteen families into
five types, according a certain metric on the set $T=(G_K\backslash M)\cup (G_L\backslash M)$.  Finally, a correspondence is exhibited between the type of a relation and the complexity (measured in terms of the number of summands, and the number of factors in each summand) of the coefficients therein.

\section{Definitions and notations}
The hypergeometric series of type ${}_{p+1}F_p$ is defined by
\begin{equation*}
\label{210}
{}_{p+1}F_p \left[ \genfrac{} {}{0pt}{}{a_1,a_2,\ldots,a_{p+1};}{
 b_1,b_2,\ldots,b_p;} z\right] =
\sum_{n=0}^{\infty} \frac{(a_1)_n(a_2)_n \cdots
(a_{p+1})_n}{n!(b_1)_n(b_2)_n \cdots (b_p)_n}z^n,
\end{equation*}
where $p \in \mathbb{Z}^{+}$, $a_1,a_2,\ldots,a_{p+1},
b_1,b_2,\ldots,b_p, z\in \mathbb{C}$, and the rising factorial
$(a)_n$ is given by
\begin{equation*}
(a)_n=\left\{ \begin{array}{ll}
a(a+1)\cdots(a+n-1), & n>0,\\
1, & n=0.
\end{array} \right.
\end{equation*}Note that, by the functional equation\begin{equation*}\Gamma(s+1)=s\Gamma(s) \quad (\hbox{Re}(s)>0)\end{equation*}for the gamma function, we have $$(a)_n=\frac{\Gamma(a+n)}{\Gamma(a)}$$ for $a\not
\in\{0,-1,-2,\ldots\}$.

The series in (\ref{210}) converges absolutely if $|z|<1$, or if
$|z|=1$ and $\textrm{Re}(\sum_{i=1}^pb_i-\sum_{i=1}^{p+1}a_i)>0$ (see
\cite[p.\ 8]{Ba}). To avoid poles, we assume that no denominator parameter
$b_1,b_2,\ldots,b_p$ is a negative integer or zero. If a numerator
parameter $a_1,a_2,\ldots,a_{p+1}$ is a negative integer or zero, the
series has only finitely many nonzero terms, and is said to 
{terminate}.

When $z=1$, we say that the series is of {\it unit argument}, or {\it of type}
${}_{p+1}F_p(1)$. If $\sum_{i=1}^pb_i-\sum_{i=1}^{p+1}a_i=1$, the
series is called {\it Saalsch\"utzian} (or, equivalently, {\it balanced}). If
$1+a_1=b_1+a_2=\cdots=b_p+a_{p+1}$, the series is called 
{\it well-poised}.
A well-poised series that satisfies $a_2=1+\frac{1}{2}a_1$ is called
{\it very-well-poised}.

The functions $K(\vec{x})=K(a;b,c,d;e,f,g)$ and $L(\vec{x})=L(a,b,c,d;e;f,g)$ studied in
\cite{FGS} and \cite{M1}, respectively, are defined by
\begin{align}\label{Kdef}
 K(\vec{x})\nr=\,&\frac{{}_4F_3\left[ \genfrac{} {}{0pt}{} {\displaystyle a,b,c,d;}{\ds
e,f,g;}1\right]}{\sin\pi a\,\Gamma(a)\Gamma(1+a-e)\Gamma(1+a-f)
\Gamma(1+a-g)\Gamma(e)
\Gamma(f)\Gamma(g)} \nr+\,&\frac{{}_4F_3\left[  \genfrac{} {}{0pt}{}{\displaystyle a,1+a-e,1+a-f,1+a-g;}{
\displaystyle 1+a-b,1+a-c,1+a-d;}1\right]} {\sin\pi a\,\Gamma(a)\Gamma(b)\Gamma(c) \Gamma(d)\Gamma(1+a-b)
\Gamma(1+a-c)\Gamma(1+a-d)}
\nr=\,&\frac{1}{\sin\pi a\,\gg{a}\gg{b}\gg{c}\gg{d}\gg{a}\gg{1+a-e}\gg{1+a-f}\gg{1+a-g}}\nr
\times\,&\biggl({_4}F_3^* \biggl[ \genfrac{} {}{0pt}{}{a,b,c,d;}{ e,f,g;}1\biggr] +{_4}F_3^* \biggl[ \genfrac{} {}{0pt}{}{a,1+a-e,1+a-f,1+a-g;}{ 1+a-b,1+a-c,1+a-d;}1\biggr]\biggr),
\end{align}
and
\begin{align}\label{Ldef}
 L(\vec{x})\nr  =\,&\frac{{}_4F_3\left[  \genfrac{} {}{0pt}{}{\displaystyle a,b,c,d;}{ \displaystyle
e,f,g;}1\right]}{\sin \pi e\,\Gamma(1+a-e)\Gamma(1+b-e)\Gamma(1+c-e)
\Gamma(1+d-e)\Gamma(e)
\Gamma(f)\Gamma(g)} \nr-\,&\frac{{}_4F_3\left[  \genfrac{} {}{0pt}{}{\displaystyle 1+a-e,1+b-e,1+c-e,1+d-e;}{
\displaystyle 2-e,1+f-e,1+g-e;}1\right]} {\sin \pi e\,
\Gamma(a)
\Gamma(b)\Gamma(c)\Gamma(d)\Gamma(2-e)\Gamma(1+f-e)\Gamma(1+g-e)}
\nr=\,&\frac{1}{ \biggl[ \genfrac{} {}{0pt}{}{\ds\sin\pi e\,\gg{a}\gg{b}\gg{c}\gg{d}\gg{1+a-e}}{\ds\times\gg{1+b-e}\gg{1+c-e}\gg{1+d-e}}\biggr]}\nr
\times\,&\biggl({_4}F_3^* \biggl[ \genfrac{} {}{0pt}{}{a,b,c,d;}{ e,f,g;}1\biggr] -{_4}F_3^* \biggl[ \genfrac{} {}{0pt}{}{\displaystyle 1+a-e,1+b-e,1+c-e,1+d-e; }{
\displaystyle 2-e,1+f-e,1+g-e;}1\biggr]\biggr),
\end{align}
where  \begin{align*}
{_4}F_3^* \biggl[ \genfrac{} {}{0pt}{}{a,b,c,d;}{ e,f,g;}1\biggr]=\,&\frac{\Gamma(a)\Gamma(b)\Gamma(c)\Gamma(d)}{
 \Gamma(e)\Gamma(f)\Gamma(g)}
{{_4}F_3 \biggl[ \genfrac{} {}{0pt}{}{a,b,c,d; }{ e,f,g;}1\biggr]}\nr=\,&\sum_{k=0}^\infty\frac
{\Gamma(k+a)\Gamma(k+b)\Gamma(k+c)\Gamma(k+d)}{
k!\Gamma(k+e)\Gamma(k+f)\Gamma(k+g)}  \end{align*} and $\vec{x}=(a,b,c,d,e,f,g)^T\in V$.

\begin{Remark}
\label{R210}  The function $K(\vec{x})$ defined  above is, in fact, $1/({\sin\pi{a}\,\gg{a}})$ times the function  $K(\vec{x})$ considered in \cite{FGS}.  The renormalization provided here will allow for greater simplicity and uniformity of the three-term relations studied in Sections 5 and 6 below.  Moreover, because all invariances (two-term relations) for $K$ described in \cite{FGS} preserve the first coordinate  of $(a,b,c,d,e,f,g)^T$, such invariances continue to hold under our new normalization.

\end{Remark}

We will call the two Saalsch\"utzian ${}_4F_3(1)$ series in 
the definition of the $K$ function 
{\it complementary} with respect 
to the parameter $a$, and will call the two 
Saalsch\"utzian ${}_4F_3(1)$ series in 
the definition of the $L$ function 
{\it supplementary} with respect 
to the parameter $e$. 

We note that, by \cite[Eq.\ $(7.5.3)$]{Ba}, the $L$ function can be
expressed as a very-well-poised ${}_7F_6(1)$ series:
\begin{align}
\label{eq240}
&L(a,b,c,d;e;f,g) \nr
\quad&=\frac{\Gamma(1+A)}{\left[ \genfrac{} {}{0pt}{}{\displaystyle
\pi\Gamma(1+A-B)\Gamma(1+A-C)\Gamma(1+A-D) }{ \displaystyle \times
\Gamma(1+A-E)\Gamma(1+A-F)\Gamma(2+2A-B-C-D-E-F)}\right]} \nr
 \times\,& {}_7F_6\left[ \genfrac{} {}{0pt}{}{\displaystyle A,1+\frac{1}{2}A,B,C,D,E,F;
}{ \displaystyle
\frac{1}{2}A,1+A-B,1+A-C,1+A-D,1+A-E,1+A-F;}1\right],
\end{align}
where
\begin{eqnarray*}
A=d+g-e, \quad B=g-a, \quad C=g-b, \\
D=g-c, \quad E=d, \quad F=1+d-e,
\end{eqnarray*}
and we require that 
$\textrm{Re}(2+2A-B-C-D-E-F)=\textrm{Re}(f-d)>0$.
The above representation of the $L$ function as a 
very-well-poised ${}_7F_6(1)$ series can also
be written as
\begin{equation*}
L(a,b,c,d;e;f,g)=\frac{\psi[A;B,C,D,E,F]}{\pi}, 
\end{equation*}
where $\psi$ is the function defined in \cite[Eqs.\ (2.1) and (2.11)]{Wh3}.
Thus, the results concerning the
$L$ function can also be interpreted in terms of the very-well-poised
${}_7F_6(1)$ series given in (\ref{eq240}). However, in the present
paper, we primarily view the $L$ function as a linear combination of
two Saalsch\"utzian ${}_4F_3(1)$ series, because this representation
connects the $L$ function more directly  to the $K$ function, which is also defined
as such a linear combination.
We also note that the very-well-poised ${}_7F_6(1)$ series version of
the $L$ function  converges only for appropriate values of the parameters,
while the Saalsch\"utzian condition $e+f+g-a-b-c-d=1$ guarantees the
convergence of both ${}_4F_3(1)$ series in the definition of the
$L$ (and the $K$) function.

\section{Previously obtained results concerning the functions 
$K(\vec{x})$ and $L(\vec{x})$}
In this section, we review the results obtained in \cite{FGS} and
\cite{M1} regarding the functions 
$K(\vec{x})$ and $L(\vec{x})$,  respectively.

Consider the matrix group $GL(7,\mathbb{C})$, acting on $\mathbb\C^7$ from the left.  If $\sigma\in S_7$, we will
identify $\sigma$ with the element of $GL(7,\mathbb{C})$ that permutes
the standard basis $\{e_1,e_2,\ldots,e_7\}$ of the complex vector space
$\mathbb{C}^7$ according to the permutation $\sigma$. For example,
\begin{equation*}
(123)=
\left( \begin{array}{ccccccc}
0 & 0 & 1 & 0 & 0 & 0 & 0 \\
1 & 0 & 0 & 0 & 0 & 0 & 0 \\ 
0 & 1 & 0 & 0 & 0 & 0 & 0 \\ 
0 & 0 & 0 & 1 & 0 & 0 & 0 \\ 
0 & 0 & 0 & 0 & 1 & 0 & 0 \\ 
0 & 0 & 0 & 0 & 0 & 1 & 0 \\ 
0 & 0 & 0 & 0 & 0 & 0 & 1         
\end{array} \right).
\end{equation*}

We will also let  $A \in GL(7,\mathbb{C})$ be the matrix
\begin{equation}
\label{310}
A=
\left( \begin{array}{ccccccc}
1 & 0 & 0 & 0 & 0 & 0 & 0 \\
0 & 1 & 0 & 0 & 0 & 0 & 0 \\ 
0 & 0 & -1 & 0 & 0 & 0 & 1 \\ 
0 & 0 & 0 & -1 & 0 & 0 & 1 \\ 
0 & 0 & -1& -1 &1&0 & 1 \\ 
0 & 0 & -1&-1 & 0&1 & 1 \\ 
0 & 0 & 0 & 0 & 0 & 0 & 1         
\end{array} \right),
\end{equation}so that, if $\vec{x}=(a,b,c,d,e,f,g)^T \in V$, then
$$\aligned A\vec{x}&=(a,b,g-c,g-d,e+g-c-d,f+g-c-d,g)^T\\&=(a,b,g-c,g-d,1+a+b-f,1+a+b-e,g)^T\in V.\endaligned$$

We consider the subgroup $M$ of $GL(7,\mathbb{C})$ defined by:
\begin{equation*}
M=\langle (12),(23),(34),(56),(67),A \rangle.
\end{equation*}
It is shown in \cite{FGS} and \cite{M1} that $M$ is
isomorphic to the Coxeter group $W(D_6)$ of order 23040.

The Coxeter group $W(D_6)$ is well-known to have a 
center consisting of two elements 
(see \cite[pp.\ 82 and 132]{Hum}). We denote by $w_0$ 
the unique nonidentity element in the center
of $M$. The element $w_0$ is called the {\it central involution},
and is computed to be
\begin{equation*}
w_0=(12)(34)[[(1234)(567)]^2A]^4. 
\end{equation*}
We observe that, if $\vec{x}=(a,b,c,d,e,f,g)^T \in V$, then
\begin{equation*}
w_0\vec{x}=
(1-a,1-b,1-c,1-d,2-e,2-f,2-g)^T. 
\end{equation*}

We next consider the following subgroups of $M$:
\begin{equation*}
G_K=\langle (23),(34),(56),(67),A \rangle
\end{equation*}
and\begin{equation*}
G_L=\langle (12),(23),(34),(67),A \rangle.
\end{equation*}Then:

\begin{enumerate}\item[(a)]
It is shown in \cite{FGS} that $G_K$ is an invariance group for 
$K(\vec{x})$; that is,
\begin{equation*}
K(\vec{x})=K(\alpha\vec{x}) \textrm{ for all } \alpha \in G_K,
\end{equation*}
and, furthermore, that $G_K$ is 
isomorphic to the symmetric group $S_6$ of
order 720. This gives 720 invariances, or 
two-term relations, for the 
function $K$. 

\item[(b)]
Similarly, it is shown in \cite{M1} that $G_L$ is an invariance group for
$L(\vec{x})$, meaning
\begin{equation*}
L(\vec{x})=L(\beta\vec{x}) \textrm{ for all } \beta \in G_L,
\end{equation*}
and, furthermore, that $G_L$ is 
isomorphic to the Coxeter group $W(D_5)$
of order 1920. This gives 1920 invariances, or 
two-term relations, for the 
function $L$.\end{enumerate}

We can express the above invariances of $K(\vec{x})$ and $L(\vec{x})$ succinctly if we reparameterize the $K$ and $L$ functions as follows.  Let 
\begin{equation}\label{Ktwiddle}
\widetilde{K}(x_0,x_1,x_2,x_3,x_4,x_5)
= K(A;B,C,D; 
 E,F,G )
\end{equation}
and
\begin{equation}\label{Ltwiddle}
\widetilde{L}(x_0,x_1,x_2,x_3,x_4,x_5)
= L(A,B,C,D; 
  E;F,G ),
\end{equation}
where
\begin{align}\label{twiddleparams}&A=\frac{1}{2}+x_0+x_1+x_2+x_3+x_4+x_5, \hskip2.25pt
B=\frac{1}{2}+x_0+x_1+x_2-x_3-x_4+x_5,\nr&C=\frac{1}{2}+x_0+x_1+x_2+x_3-x_4-x_5, \hskip2.25pt
D=\frac{1}{2}+x_0+x_1+x_2-x_3+x_4-x_5,\nr&E=1+2x_1+2x_2, \ 
F=1+2x_0+2x_1, \ 
G=1+2x_0+2x_2.\end{align}(This reparameterization of the $L$ function is equivalent to the
reparameterization given in \cite[equation (3.1)]{Wh3}.)
Then the following is readily shown: 

\begin{enumerate}
\item[(a)] The invariance of $K(\vec{x})$ under the group
$G_K\cong  S_6$  amounts to the
equivalence of the 720 functions
of the form  
\begin{align*}\{&\widetilde{K}(x_{i_0},x_{i_1},x_{i_2},x_{i_3},x_{i_4},x_{i_5})\colon (i_0,i_1,i_2,i_3,i_4,i_5)  \\&\hbox{is a permutation of }
(0,1,2,3,4,5)\}.\end{align*}

\item[(b)] The invariance of $L(\vec{x})$ under the group
$G_L\cong  W(D_5)$  amounts to the
equivalence of the 1920 functions \begin{align*}\{& \widetilde{L}(x_0,\pm x_{i_1},\pm x_{i_2},\pm x_{i_3},\pm x_{i_4},\pm x_{i_5})\colon (i_1,i_2,i_3,i_4,i_5)  \hbox{ is a permutation}\\&\hbox{ of }
(1,2,3,4,5)\hbox{ and the number of negative signs is even}\}.\end{align*} 
\end{enumerate}

The three-term relations for $K(\vec{x})$ and
$L(a,b,c,d;e;f,g)$ are governed by 
the right cosets of their invariance
groups $G_K$ and $G_L$, respectively, in $M$, as follows.

First we discuss the three-term relations
for $K(\vec{x})$.
The number of right cosets of $G_K$ in $M$ is 
${|M|}/{|G_K|}={23040}/{720}=32$. In \cite{FGS}, these cosets are indexed in the following way:  consider the transformation $R$ of $\C^4$ defined by  $$R(x,y,z,t)=(y,z,t,x).$$  Denote the images of the vectors $(a,b,c,d)$ and $(1,e,f,g)$ under the $j$th power transformation $R^j$ by  $(a_j,b_j,c_j,d_j)$ and $(1_j,e_j,f_j,g_j)$ respectively.  Also, for  $k \in \{0,1, \ldots, 15\}$, write $k=4q+r$ with $0\le q,r\le 3$. Then we denote by $p_k$ the element of $G_K\backslash M$ that contains the transformation \begin{equation}\label{pkdef}\vec{x}\to \biggl( \genfrac{} {}{0pt}{}{1+a_r-1_q,1+b_r-1_q,1+c_r-1_q,1+d_r-1_q, 
}{ 1+e_q-1_q,1+f_q-1_q,1+g_q-1_q}\biggr)^T\in V,\end{equation} and by $n_k$ the element of $G_K\backslash M$ that contains the transformation  \begin{equation}\label{nkdef}\vec{x}\to \biggl( \genfrac{} {}{0pt}{}{1_q-a_r,1_q-b_r,1_q-c_r,1_q-d_r ,
}{ 1+1_q-e_q,1+1_q-f_q,1+1_q-g_q }\biggr)^T\in V.\end{equation}
Note that, for each 
$k \in \{0,1, \ldots, 15\}$, $p_k$ and $n_k$ are interchanged under the action
of the central involution $w_0$. 

\begin{Definition}
\label{hamming}
We define the 
{\it Hamming distance}  $d(\sigma,\tau)$ between elements 
$\sigma,\tau\in G_K\backslash M$ by 
\begin{align}d(p_{4q+r},p_{4s+t})=  d(n_{4q+r},n_{4s+t})=\,&4-2(\delta_{q,s}+\delta_{r,t}),\nr d(p_{4q+r},n_{4s+t})=d(n_{4q+r},p_{4s+t})=\,&2+2(\delta_{q,s}+\delta_{r,t}),\nonumber\end{align}for $q,r,s,t\in\{0,1,2,3\}$, where $\delta_{q,s}$ denotes the Kronecker delta.
\end{Definition}

\begin{Remark} The definition of Hamming distance given in \cite{FGS} appears, {\it a priori}, quite different from the one given above, but, as may readily be checked, the two definitions are equivalent.\end{Remark}

For example, $$d(p_3,n_8)=2+2(\delta_{0,2}+\delta_{3,0})=2.$$Note also that, for any $k \in \{0,1, \ldots, 15\}$,  $d(p_k,n_k)=6$.

It is shown in \cite[Lemma $6.3$]{FGS} that the right coset
space $G_K \backslash M$ is a metric space with respect
to Hamming distance, and that
the action of $M$ on $G_K\backslash M$ by right multiplication is an isometry with respect to this  distance.

\begin{Definition}
\label{D320}
Let $S(K^3)$ denote the set of three-element subsets of  $G_K\backslash M$. The 
{\it Hamming type} of  
$\{\sigma,\tau,\mu\}\in  S(K^3)$ is 
defined to be the three-digit integer
$abc$, where $a$ is the shortest among the Hamming distances
  $d(\sigma,\tau)$,   $d(\sigma,\mu)$, and   $d(\tau,\mu)$; $b$ 
is the next shortest;
and $c$ is the longest. 
\end{Definition}

Note that $|S(K^3)|=\binom{32}{3}=4960$. It is shown in  \cite[Proposition 6.5]{FGS} 
that each element of $S(K^3)$ has Hamming type $222,224,244,246$, or $444$, and that the action of $M$ on $S(K^3)$ by right multiplication (elementwise) partitions $S(K^3)$
into five orbits, with each orbit consisting of all elements of  a given Hamming type.

Write $K_\sigma(\vec{x})$ for $K(\alpha\vec{x})$ whenever  $\alpha\in M$ belongs to the right coset $\sigma\in G_K\backslash M$. A  three-term relation among the functions
$K_{\sigma} , K_{\tau}$, and 
$K_{\mu}$
is called an $abc$ {\it relation}
if  
$\{\sigma,\tau,\mu\}$ is of Hamming
type $abc$. Explicit three-term relations for each of the Hamming types $222,224, 244, 246$, and $444$
 are obtained
in Propositions $7.3$, $7.4$, $7.5$,
$7.6$,  $7.7$, respectively, of \cite{FGS}.  Applying the action of $M$ to the $K$ functions and the coefficients of these five relations, we thereby obtain 4960 three-term relations, partitioned by Hamming type into five families.
\begin{Remark}
\label{R315} Let $\widetilde{K}$ be as defined in equations (\ref{Ktwiddle}) and (\ref{twiddleparams}) above.  Then the set of 23040 $K$ functions that relate to each other via the three-term relations just described is equal to the set
\begin{align*}\{&\widetilde{K}(\pm x_{i_0},\pm x_{i_1},\pm x_{i_2},\pm x_{i_3},\pm x_{i_4},\pm x_{i_5})\colon 
 (i_0,i_1,i_2,i_3,i_4,i_5)\hbox{ is a permutation}\nr
&\hbox{of }(0,1,2,3,4,5)\hbox{ and the number of negative signs is even}\}.\end{align*}
Moreover, the set 
$$\{\widetilde{K}( \pm x_0,\pm x_1,\pm x_2,\pm x_3,\pm x_4,\pm x_5)\colon \hbox{the number of negative signs is even}\}$$equals the set$$\{K_\sigma(\vec{x})\colon \sigma\in G_K\backslash M\}.$$
\end{Remark}
We turn next to a discussion of three-term relations for $L(\vec{x})$.
The number of right cosets of $G_L$ in $M$ is 
${|M|}/{|G_L|}={23040}/{1920}=12$. The 12 right cosets
of $G_L$ in $M$ are labeled in \cite{M1} by 
$1,2,\ldots,6,\overline{1},\overline{2},\ldots,\overline{6}$, where,  for
each $i \in \{1,2,\ldots,6\}$,   $i $ and ${\overline{i}}$ are
interchanged under the action of the central involution $w_0$. In particular, the corresponding $L$ functions are defined as follows:

\begin{align}\label{Lcosetsdef}
& L_6(\vec{x})=L\left(a,b,c,d;   e;f,g \right),\nr&
 L_5(\vec{x})=L\left(a,b,c,d;  f;e,g \right),\nr&
 L_4(\vec{x})=L\left(a,b,c,d;  g;f,e \right),\nr &
 L_3(\vec{x})=L\left(a,1+a-e,1+a-f,1+a-g;  1+a-b;1+a-c,1+a-d \right),\nr
&L_2(\vec{x})=L\left(a,1+a-e,1+a-f,1+a-g;  1+a-c;1+a-b,1+a-d \right),\nr&
L_1(\vec{x})=L\left(a,1+a-e,1+a-f,1+a-g;  1+a-d;1+a-c,1+a-b \right),\nr 
&L_{\overline{6}}(\vec{x})=L\left(1-a,1-b,1-c,1-d;  2-e;2-f,2-g \right),\nr&  L_{\overline{5}}(\vec{x})=L\left(1-a,1-b,1-c,1-d;  2-f;2-e,2-g \right),\nr
&L_{\overline{4}}(\vec{x})=L\left(1-a,1-b,1-c,1-d;  2-g;2-f,2-e \right),\nr&
 L_{\overline{3}}(\vec{x})=L\left(1-a,e-a,f-a,g-a;  1+b-a;1+c-a,1+d-a \right),\nr
&L_{\overline{2}}(\vec{x})=L\left(1-a,e-a,f-a,g-a;  1+c-a;1+b-a,1+d-a \right),\nr
&L_{\overline{1}}(\vec{x})=L\left(1-a,e-a,f-a,g-a;  1+d-a;1+c-a,1+b-a \right).
\end{align} 

We have

\begin{Definition}
\label{coherdef}
A subset of $G_L \backslash M$ is  is called 
{\it $L$-coherent}
if no two elements of this subset are interchanged by the action of the central
involution $w_0$. A subset of $G_L \backslash M$ that is not $L$-coherent  is called 
{\it $L$-incoherent}.  
\end{Definition}

We note that a subset of $G_L\backslash M$ is   $L$-coherent if and only if it does not
contain both elements of the form $i$ and $\overline{i}$, for any
$i \in \{1,2, \ldots, 6\}$.

The group $M$ acts by right multiplication (elementwise) on the set $S(L^3)$ of 
three-element subsets of
$G_L \backslash M$.
There are $\binom{12}{3}=220$ 
such subsets. It is shown in 
\cite[Proposition $6.6$]{M1} that this group action partitions $S(L^3)$ into two orbits---an orbit of length 160,
 consisting of the $L$-coherent elements of $S(L^3)$, and an  orbit of
length 60, consisting of the $L$-incoherent elements.

A three-term relation among the functions
$L_i(\vec{x}), L_j(\vec{x})$, and $L_k(\vec{x})$, for $i,j,k\in \{1,2,\ldots,6,\overline1,\overline2,\ldots,\overline6\}$, is said to
be {\it $L$-coherent} if  $\{i,j,k\}$ is an $L$-coherent set, and {\it $L$-incoherent} otherwise.
Explicit examples of $L$-coherent and $L$-incoherent relations are given in  Propositions $7.2$ and  $7.4$, respectively, of \cite{M1}.
  Applying the action of $M$ to the $L$ functions and the coefficients of these two relations, we thereby obtain 220 three-term relations, partitioned by coherence into two families.

\begin{Remark}
Let $\widetilde{L}$ be as defined in equations (\ref{Ltwiddle}) and (\ref{twiddleparams}) above.  Then the set of 23040 $L$ functions that relate to each other via the three-term relations just described is equal to the set
\begin{align*}\{&\widetilde{L}(\pm x_{i_0},\pm x_{i_1},\pm x_{i_2},\pm x_{i_3},\pm x_{i_4},\pm x_{i_5})\colon 
 (i_0,i_1,i_2,i_3,i_4,i_5)\hbox{ is a permutation}\nr
&\hbox{of }(0,1,2,3,4,5)\hbox{ and the number of negative signs is even}\}.\end{align*}
Moreover,  the set
$$\{\widetilde{L}(  \pm x_i, x_{j(i)}, x_{k(i)}, x_{\ell(i)}, x_{m(i)}, x_{n(i)})\colon 0\le i\le 5 \},$$where, for a given $i\in\{0,1,2,3,4,5\}$,    $j(i),k(i),\ell(i),m(i)$, and $n(i)$ denote the distinct elements of $\{0,1,2,3,4,5\}\backslash \{i\}$ (in some fixed order), equals the set$$\{L_\sigma(\vec{x})\colon \sigma\in G_L\backslash M\}.$$
\label{R320} 
\end{Remark}

\section{Coxeter group actions and orbits}
In this section, we describe the group-theoretic structure
behind the $(K,L,L)$ and $(L,K,K)$ relations. To this end, we
study the orbits of the action by right multiplication (elementwise) of the
group $M$ on the sets \begin{align*}S(K,L^2)=G_K\backslash M \times S(L^2) \hbox{ and }S(L,K^2)=G_L\backslash M \times S(K^2).\end{align*} (Here, $S(L^2)$ denotes the set of two-element subsets of $G_L\backslash M$, and similarly for $S(K^2)$.)
All  group actions in this section are assumed to
be right group actions.

The Dynkin diagram of the Coxeter group $W(D_n)$ is given by
the graph with vertices labeled $1',1,2,\ldots,n-1$, where
$i,j \in \{1,2,\ldots,n-1\}$ are connected by an edge if and only if
$|i-j|=1$, and $1'$ is connected to $2$ only. The presentation
of $W(D_n)$ is given by
$$W(D_n)=\langle s_{1'},s_1,s_2,\ldots,s_{n-1}:(s_i s_j)^{m_{ij}}=1 \rangle,$$
where $m_{ii}=1$ for all $i$; and for $i$ and $j$ distinct, $m_{ij}=3$
if $i$ and $j$ are connected by an edge, and $m_{ij}=2$ otherwise.
It is well-known that the order of $W(D_n)$ is
$2^{n-1}n!$ (see \cite[Section $2.11$]{Hum}).

The Coxeter generators 
$s_i, i \in \{1',1,2,3,4,5\}$,
of the group $M$, which is isomorphic to
$W(D_6)$, are shown in \cite{M1} to be
$$s_1=(34), s_2=(23), s_3=(34)A, s_4=(67),
s_5=(56), s_{1'}=(12),$$
where $A$ is the matrix given in (\ref{310}).

The invariance groups $G_K$ and $G_L$ for 
the functions
$K(\vec{x})$ and
$L(\vec{x})$, respectively, are given by
$$G_K=\langle s_1,s_2,s_3,s_4,s_5 \rangle \cong S_6$$
and
$$G_L=\langle s_{1'},s_1,s_2,s_3,s_4 \rangle \cong W(D_5).$$

Consider the action of $M$ on
$G_L \backslash M = \{1,\ldots,6,\overline{1},\ldots,\overline{6}\}$.
If $t \in M$, then $t$ induces a permutation on the 
twelve-element set 
$\{1,\ldots,6,\overline{1},\ldots,\overline{6}\}$. Let
$\Phi:M \to S_{12}$ be the induced permutation
representation. For every $t \in M$, the
permutation $\Phi(t)$ can be uniquely described by
specifying its effect on the set
$\{1,\ldots,6\}$, since the elements in
$\{\overline{1},\ldots,\overline{6}\}$ will be permuted based on
where $\{1,\ldots,6\}$ are permuted with the addition
or omission of a bar.
The proposition given below is the
result in Proposition $6.3$ from \cite{M1}:

\begin{Proposition}
\label{P410}
The images of the generators 
$s_1,s_2,s_3,s_4,s_5$, and $s_{1'}$ of $M$ under the 
permutation representation 
$\Phi:M \to S_{12}$
are given by
\begin{eqnarray*}
&&\Phi(s_1)=\Phi\left((34)\right)
=\left( \begin{array}{cccccc}
1 & 2 & 3 & 4 & 5 & 6 \\
2 & 1 & 3 & 4 & 5 & 6
\end{array} \right), \\
&&\Phi(s_2)=\Phi\left((23)\right)
=\left( \begin{array}{cccccc}
1 & 2 & 3 & 4 & 5 & 6 \\
1 & 3 & 2 & 4 & 5 & 6
\end{array} \right), \\
&&\Phi(s_3)=\Phi\left((34)A\right)
=\left( \begin{array}{cccccc}
1 & 2 & 3 & 4 & 5 & 6 \\
1 & 2 & 4 & 3 & 5 & 6
\end{array} \right), \\
&&\Phi(s_4)=\Phi\left((67)\right)
=\left( \begin{array}{cccccc}
1 & 2 & 3 & 4 & 5 & 6 \\
1 & 2 & 3 & 5 & 4 & 6
\end{array} \right), \\
&&\Phi(s_5)=\Phi\left((56)\right)
=\left( \begin{array}{cccccc}
1 & 2 & 3 & 4 & 5 & 6 \\
1 & 2 & 3 & 4 & 6 & 5
\end{array} \right), \\
&&\Phi(s_{1'})=\Phi\left((12)\right)
=\left( \begin{array}{cccccc}
1 & 2 & 3 & 4 & 5 & 6 \\
\overline{2} & \overline{1} & 3 & 4 & 5 & 6
\end{array} \right). 
\end{eqnarray*}
\end{Proposition}

A further description of the permutation representation
$\Phi:M \to S_{12}$
is given in the proposition below, which is the
result in Proposition $6.4$ of \cite{M1}:

\begin{Proposition}
\label{P420}
The permutation representation
$\Phi:M \to S_{12}$ is faithful, i.e.
$\textrm{ker}(\Phi)=\{I_7\}$. The embedding of $M$ into
$S_{12}$ consists of all the permutations of the form
$\left( \begin{array}{cccccc}
1 & 2 & 3 & 4 & 5 & 6 \\
j_1 & j_2 & j_3 & j_4 & j_5 & j_6
\end{array} \right),$ where each $j_i$ belongs
to the set $\{1,\ldots,6,\overline{1},\ldots,\overline{6}\}$,
the $j_i$'s are all distinct if we remove the bars,
and an even number, i.e. $0,2,4$, or $6$ of the
$j_i$'s contain a bar.
\end{Proposition}

The two main results concerning the 
orbits of the action by right multiplication of the
group $M$ on the sets 
$S(K,L^2)$
and
$S(L,K^2)$
will be given in Proposition \ref{P430} and Proposition \ref{P440},
respectively. Before we state and prove those 
propositions, we
need the following lemma:

\begin{Lemma}
\label{L410}
Let $G$ be a finite group acting on the finite sets $A$ and $B$. Let
$\{\alpha_1, \alpha_2, \ldots, \alpha_n\}$ be 
a complete set of orbit representatives of
$G$ acting on $A$. 
Let $G_{\alpha_1}, G_{\alpha_2}, \ldots, G_{\alpha_n}$ be the 
corresponding stabilizers of $\alpha_1, \alpha_2, \ldots, \alpha_n$
in $G$. For each $i$ such that $1 \leq i \leq n$, let 
$\{\beta_{i,1}, \beta_{i,2}, \ldots, \beta_{i,m_i}\}$ 
be a complete set of orbit representatives
of $G_{\alpha_i}$ acting on $B$. Then there are
$\sum_{i=1}^n m_i$ orbits of the 
action of $G$ on $A \times B$ and
a complete set of orbit representatives is given by
$I=\cup_{i=1}^n \cup_{j=1}^{m_i} \{(\alpha_i,\beta_{i,j})\}$.
Furthermore, for any $(\alpha_i,\beta_{i,j}) \in I$,
the length of the orbit containing $(\alpha_i,\beta_{i,j})$
under the action of $G$ on $A \times B$
is equal to the length of the orbit containing $\alpha_i$ under
the action of $G$ on $A$ multiplied by the length of the orbit
containing $\beta_{i,j}$ under the action of $G_{\alpha_i}$
on $B$.
\end{Lemma}

\begin{proof}
Take any $(\alpha,\beta) \in A \times B$. Then there exist
$i \in \{1,2,\ldots,n\}$ and $g \in G$ such that
$\alpha g = \alpha_i$. Furthermore, there exist
$j \in \{1,2,\ldots,m_i\}$ and $h \in G_{\alpha_i}$
such that $(\beta g) h = \beta_{i,j}$. Hence
$(\alpha,\beta) gh = (\alpha_i,\beta_{i,j})$ and so any element
of $A \times B$ is in an orbit whose representative is an
element of $I$.

Next suppose that for some 
$(\alpha_i,\beta_{i,j}), (\alpha_s,\beta_{s,t}) \in I$, there exists
$g \in G$ such that 
$(\alpha_i,\beta_{i,j})g= (\alpha_s,\beta_{s,t})$.
Then $\alpha_i g = \alpha_s$ and, 
since $\alpha_1, \alpha_2, \ldots, \alpha_n$
are representatives of different orbits of $G$ acting
on $A$, we must have $\alpha_i=\alpha_s$
and $g \in G_{\alpha_i}$.
Hence, since  
$\beta_{i,j} g = \beta_{s,t}$, we have
$\beta_{i,j} g = \beta_{i,t}$. But 
$\beta_{i,1}, \beta_{i,2}, \ldots, \beta_{i,m_i}$
are representatives of different orbits of
$G_{\alpha_i}$ acting on $B$,  so we must have
$\beta_{i,j} = \beta_{i,t}$. Therefore 
$(\alpha_i,\beta_{i,j}) = (\alpha_s,\beta_{s,t})$, which shows that
every two elements of $I$
are representatives of different orbits. We conclude that
$I$ is a complete set of orbit representatives as claimed. From this it also follows that there are $\sum_{i=1}^n m_i$ orbits of the 
action of $G$ on $A \times B$.

Finally, take any $(\alpha_i,\beta_{i,j}) \in I$. Let
$G_{(\alpha_i,\beta_{i,j})}$ be the stabilizer
of $(\alpha_i,\beta_{i,j})$ in $G$. 
By the Orbit-Stabilizer Theorem, the length of the orbit
containing $(\alpha_i,\beta_{i,j})$
under the action of $G$ on $A \times B$
is equal to the index of $G_{(\alpha_i,\beta_{i,j})}$
in $G$. Now since
$G_{(\alpha_i,\beta_{i,j})} \leq G_{\alpha_i} \leq G$,
the index of $G_{(\alpha_i,\beta_{i,j})}$
in $G$ is equal to the index of $G_{\alpha_i}$
in $G$ multiplied by the index of $G_{(\alpha_i,\beta_{i,j})}$
in $G_{\alpha_i}$. The index of $G_{\alpha_i}$
in $G$ is equal to the length of the orbit 
containing $\alpha_i$ under
the action of $G$ on $A$. Furthermore, 
$G_{(\alpha_i,\beta_{i,j})}$ is also the stabilizer 
in $G_{\alpha_i}$ of $\beta_{i,j}$ 
under the action of $G_{\alpha_i}$ on $B$,
and hence
the index of $G_{(\alpha_i,\beta_{i,j})}$
in $G_{\alpha_i}$ is equal to the length of the orbit
containing $\beta_{i,j}$ under the action of $G_{\alpha_i}$
on $B$. This completes the proof of the lemma.
\end{proof}

\begin{Proposition}
\label{P430}
There are four orbits of the action of $M$ on 
$ S(K,L^2)$.
Representatives and lengths of the orbits are given by

Orbit $1$: Representative $(p_0,\{6,5\})$, length$=480$,

Orbit $2$: Representative $(p_0,\{\overline{6},5\})$, length$=960$,

Orbit $3$: Representative $(p_0,\{6,\overline{6}\})$, length$=192$, and

Orbit $4$: Representative $(p_0,\{\overline{1},\overline{2}\})$, length$=480$.
\end{Proposition}

\begin{proof}
The group $M$ acts transitively on the $32$-element set
$G_K \backslash M$.
The stabilizer $M_{p_0}$ of the element
$p_0 \in G_K \backslash M$ is generated by
$s_1,s_2,s_3,s_4,s_5$ and is isomorphic to $S_6$. 
By Lemma \ref{L410}, we need to consider the action of $M_{p_0}$ on
$S(L^2)$. Since no orbit of this action contains both an $L$-coherent element of $S(L^2)$ and an $L$-incoherent element, and since, by
Proposition \ref{P410}, every element of $M_{p_0}$ simply
permutes the elements of $\{1,\ldots,6\}$, 
there are four orbits of the action of $M_{p_0}$
on 
$S(L^2)$, and they are given by
\begin{eqnarray*}
&&\{\{i,j\}:i,j \in \{1,\ldots,6\}, i \neq j\}, \\
&&\{\{i,j\}:i \in \{\overline{1},\ldots,\overline{6}\}, j \in \{1,\ldots,6\}, \overline{j} \neq i\}, \\
&&\{\{i,\overline{i}\}:i \in \{1,\ldots,6\}\}, \mbox{ and}\\
&&\{\{i,j\}:i,j \in \{\overline{1},\ldots,\overline{6}\}, i \neq j\}. 
\end{eqnarray*}
Representatives for the above orbits are
$\{6,5\}, \{\overline{6},5\}, \{6,\overline{6}\}$, and $\{\overline{1},\overline{2}\}$,
respectively.
A simple counting argument shows that the lengths of
those orbits are
${(6 \times 5)}/{2} = 15, 6 \times 5 = 30, 6$, and $({6 \times 5})/{2} = 15$,
respectively.

Further, by Lemma \ref{L410}, there are four orbits of the action of $M$ on  
$S(K,L^2),$
with representatives being 
$(p_0,\{6,5\}),(p_0,\{\overline{6},5\}),(p_0,\{6,\overline{6}\})$,
and $(p_0,\{\overline{1},\overline{2}\})$.
The lengths of
the orbits are 
$32 \times 15 = 480, 32 \times 30 = 960,32 \times 6 = 192$, and $32 \times 15 = 480$,
respectively.
\end{proof}

\begin{Proposition}
\label{P440}
There are seven orbits of the action of $M$ on 
$ S(L,K^2)$.
Representatives and lengths of the orbits are given by

Orbit $1$: Representative $(3,\{p_0,p_1\})$, length$=960$,

Orbit $2$: Representative $(\overline{6},\{p_0,p_1\})$, length$=960$,

Orbit $3$: Representative $(\overline{1},\{p_0,p_1\})$, length$=960$,

Orbit $4$: Representative $(4,\{p_0,n_4\})$, length$=480$, 

Orbit $5$: Representative $(\overline{4},\{p_0,n_4\})$, length$=480$,

Orbit $6$: Representative $(\overline6,\{p_0,n_4\})$, length$=1920$, and

Orbit $7$: Representative $(6,\{p_0,n_0\})$, length$=192$. 
\end{Proposition}

\begin{proof}
By the results in \cite{FGS}, there are three orbits of the action 
of $M$ on  
$S(K^2)$. Each of these orbits
consists of all  $\{\sigma,\tau\}\subset S(K^2)$ such that the Hamming distance $d(\sigma,\tau)$ equals
$2,4$, or $6$. Representatives of these three orbits are
$\{p_0,p_1\},\{p_0,n_4\}$, and $\{p_0,n_0\}$,
respectively.
The lengths of the three orbits are
$({32\times\binom{6}{2}})/{2}=240,
({32\times\binom{6}{4}})/2{2}=240$, 
and $16$, respectively.

The stabilizer  $M_{\{p_0,p_1\}}$ in $M$ of $\{p_0,p_1\}$
is generated by $s_{1'},s_1,s_3,s_4,s_5$ and  is
isomorphic to $S_4 \times S_2 \times S_2$. There are
three orbits of the action of $M_{\{p_0,p_1\}}$ on
$G_L \backslash M$,and the three orbits are given by
$\{3,4,5,6\},\{\overline{3},\overline{4},\overline{5},\overline{6}\},
\{1,\overline{1},2,\overline{2}\}$.

The stabilizer  $M_{\{p_0,n_4\}}$ in $M$ of $\{p_0,n_4\}$
is generated by $s_1,s_2,s_4,\lambda_1,\lambda_2$,
where
\begin{eqnarray*}
&&\lambda_1=s_4s_5s_4s_3s_4s_5s_4, \\
&&\lambda_2=s_1s_2s_1s_{1'}s_2s_1s_5s_4s_3s_2s_1s_{1'}s_2s_3s_4s_5s_4.
\end{eqnarray*}
If $\vec{x}=(a,b,c,d,e,f,g)^T \in V$, then
the actions of $\lambda_1$ and $\lambda_2$ 
on $\vec{x}$ are given by
\begin{equation*}
\lambda_1\vec{x}=
(a,b,e-d,e-c,e,1+a+b-g,1+a+b-f)^T
\end{equation*}
and
\begin{equation*}
\lambda_2\vec{x}=
(e-a,e-b,e-c,e-d,e,1+e-f,1+e-g)^T. 
\end{equation*}
The actions of $\lambda_1$ and $\lambda_2$
on $G_L \backslash M$ are given by
$\left( \begin{array}{cccccc}
1 & 2 & 3 & 4 & 5 & 6 \\
1 & 2 & 6 & 4 & 5 & 3
\end{array} \right)$
and
$\left( \begin{array}{cccccc}
1 & 2 & 3 & 4 & 5 & 6 \\
\overline{1} & \overline{2} & \overline{3} & 5 & 4 & \overline{6}
\end{array} \right)$, respectively.
The group $M_{\{p_0,n_4\}}$
is isomorphic to $S_4 \times S_2 \times S_2$.
There are three
orbits of the action of $M_{\{p_0,n_4\}}$ on
$G_L \backslash M$, and the three orbits are given by
$\{4,5\}, \{\overline{4},\overline{5}\},
\{1,\overline{1},2,\overline{2},3,\overline{3},6,\overline{6}\}$.

The stabilizer  $M_{\{p_0,n_0\}}$ in $M$ of $\{p_0,n_0\}$
is generated by $s_1,s_2,s_3,s_4,s_5,w_0$,
where $w_0$ is the central involution of $M$.
The group $M_{\{p_0,n_0\}}$
is isomorphic to $S_6 \times S_2$. The action
of $M_{\{p_0,n_0\}}$ on $G_L \backslash M$
is transitive and thus we have only one orbit.

An application of Lemma \ref{L410} now leads to the result
in Proposition \ref{P440}.
\end{proof}

\section{The eighteen three-term relations}

In this section and the next, we present our main results concerning three-term relations among the functions $K(\vec{x})$ and $L(\vec{x})$.

To this end, we make several definitions.  The first of these describes certain quantities in terms of which all coefficients in all three-term relations will ultimately be expressed.

\begin{Definition}  
\label{D510}

\begin{enumerate}\item[\rm(a)] For $a_1,a_2,\ldots,a_n\in \C$ and $n\in\mathbb{Z}^+$, we write  
$$S(a_1,a_2,\ldots,a_n)=\sin\pi{a_1}\sinpi{a_2}\cdots\sinpi{a_n}.$$
\item[\rm(b)]
For $\vec{x}\in V$, we define:
\begin{align*}A(\vec{x})=\,&\frac{1}{\Gamma(a)\Gamma(b)
\Gamma(c)\Gamma(d)}\quad\hbox{(also often denoted $A(a,b,c,d)$)}; \nr B_1(\vec{x}) =\,&{\pi^{-4}} {S(b,c)}\bigl[{S( b-a,c-a,d)  +S(e-a,f-a,g-a)}\bigr] ;\nr B_2(\vec{x})=\,&{\pi^{-4}}{S(f-a,g-a)}\\\times\,&\bigl[{S(b-a,c-a,d-a)+S(e-2a,f-a,g-a)}\bigr]; \nr
 B_3(\vec{x})=\,&\frac{1 }{\pi^4S(f-g)}    \biggl[\genfrac{} {}{0pt}{} {S(f,f-e,g-a,g-b,g-c,g-d) }{ -S(g,g-e,f-a,f-b,f-c,f-d)} \biggr];\nr  B_4(\vec{x}) =\,& {\pi^{-4}}{S(f-a,g-a)}\bigl[{S(b,c,d) +S(e-a,f,g)}\bigr] ;\nr B_5(\vec{x})=\,& {\pi^{-4}}{S(b,e-b)}\bigl[{S(b-a,c-a,d-a)+S(e-2a,f-a,g-a)}\bigr];  \nr C(\vec{x})= \,&\frac{S(e-a,f-a,g-a)}{\pi^8}\biggl[S(b,c,d)\biggl(\genfrac{} {}{0pt}{}{S(b-a,c-a,d-a)}{+S(e-2a,f-a,g-a)}\biggr)\\+\,&S(e,f-a,g-a)\biggl(\genfrac{} {}{0pt}{}{S(b,c,d)}{+S(e-a,f,g)}\biggr)\biggr]\\=\,&\pi^{-4}{S(e-a)}\bigl[S(b,c,d)B_2(\vec{x})+S(e,f-a,g-a)B_4(\vec{x})\bigr].\end{align*}
\item[\rm(c)]  We definite the {\it length} of any product of the above functions $S$, $A$, $B_k$ ($k\in\{1,2,3,4,5\}$), and $C$ to be the number of gamma functions  appearing in the denominator  of each summand of that product.  Here, we agree that any sine function in the numerator increases the length by two (since $\sin\pi a=\pi/(\Gamma(a)\Gamma(1-a)$), and that any sine function in the denominator  decreases the length by two.

\item[\rm(d)]  We define the {\it width} of any product of the above functions $S$, $A$, $B_k$ ($k\in\{1,2,3,4,5\}$), and $C$ to be the number of summands comprising that product.    \end{enumerate}
\end{Definition} 
  
In particular,   $S(x_1,x_2,\ldots,x_n)$ has length $2n$ and width $1$,  $A(\vec{x})$  has length $4$ and width $1$, each $B_k(\vec{x})$ ($k\in\{1,2,3,4,5\}$) has length $10$ and width $2$, and $C(\vec{x})$ has length $18$ and width $4$.

Note also that, while $B_3(\vec{x})$ (alone among the above coefficients) has a sine function in its denominator, the resulting singularities in $B_3(\vec{x})$ are removable (since the zeros of $\sin\pi(f-g)$ are also zeros of the numerator of $B_3(\vec{x})$).

Let us now define \begin{equation*}T=(G_L\backslash M) \cup (G_K\backslash M).\end{equation*}  (Note that $|T|=12+32=44$.)  We extend the notions of Hamming distance and Hamming type on $G_K\backslash M$, cf. Definitions \ref{hamming} and \ref{D320} above, to obtain notions of distance and type on $T$, as follows.

\begin{Definition}  
\label{D520} \begin{enumerate}\item[\rm(a)]
If $\sigma,\tau\in T$, then the  {\it distance} $d(\sigma,\tau)$ between $\sigma$ and $\tau$ is defined to be

\begin{enumerate}\item[\rm(i)]  $0$ if $\sigma=\tau$.

\item[\rm(ii)] $2$  if:  $\sigma,\tau\in G_K\backslash M$ and the Hamming distance  between $\sigma$ and $\tau$ is 2, {or}   (in the case where either $\sigma$ or $\tau$ belongs to $G_L\backslash M$) $\sigma$ and $\tau$ are not opposite. Here, two cosets in $T$ are said to be {\it opposite}  if some element of the former coset equals the central involution times some element of the latter.  Otherwise,  they are not opposite.

 \item[\rm(iii)]  $4$  if:  $\sigma,\tau\in G_K\backslash M$ and the Hamming distance  between $\sigma$ and $\tau$ is 4, {or}   (in the case where either $\sigma$ or $\tau$ belongs to $G_L\backslash M$) $\sigma$ and $\tau$ are  opposite. 

\item[\rm(iv)] $6$ if  $\sigma,\tau\in G_K\backslash M$ and the Hamming distance  between $\sigma$ and $\tau$ is 6.
 \end{enumerate}
\item[\rm(b)] By the {\it type} of  a three-element subset $\{\sigma,\tau,\mu\} $ of $T$, we  mean the symbol $abc$, where $a$, $b$, and $c$ are the integers $d(\sigma,\tau)$, $d(\sigma,\mu)$, and $d(\tau,\mu)$, written in weakly increasing order.  

\end{enumerate}
\end{Definition}

For example, consider the set $\{p_0,\overline{1},\overline{2}\},$ and refer to the characterizations of $G_K\backslash M$ and $G_L\backslash M$ given in Section 3 (see (\ref{pkdef}), (\ref{nkdef}), and (\ref{Lcosetsdef}))  above.  On the one hand, $\overline{1}$  and  $\overline{2}$ are not opposite because, as follows from Definition \ref{coherdef} above,  $i,j\in G_L\backslash M$ are opposite if and only if $\{i,j\}=\{k,\overline{k}\}$ for some $k\in\{1,2,3,4,5,6\}$.  On the other hand, $p_0$ is opposite  $\overline1$, because the transformation $\rho\in \overline1$ defined by
$$\rho\vec{x}= (1-a,e-a,f-a,g-a,  1+d-a,1+c-a,1+b-a )^T$$  equals $w_0\nu$, where $\nu\in p_0$ is defined by $$\nu\vec{x}= (a,1+a-e,1+a-f,1+a-g,1+a-d,1+a-c,1+a-b )^T.$$(It is clear from Definition \ref{Kdef} that $K(\vec{x})$ is invariant under $\nu$.)  
Similarly, $p_0$ is opposite $\overline{2}$.  So $\{p_0,\overline{1},\overline{2}\}$ has type $244$.
\begin{Remark}  
\label{R510}
$d(\sigma,\tau)$ defines a metric on $T$. Indeed, the only way integers $a,b,c\in\{2,4,6\}$ can fail to satisfy  $a\le b+c$ is when $a=6$ and  $b=c=2$.  But it follows from Definition \ref{D520}, and from Proposition 6.5(iv) of \cite{FGS}, that a set $\{\sigma,\tau,\mu\}\subset T$ must have type $222, 224, 244, 246$, or $444$.

It  is readily checked that $d(\sigma,\tau)$ is invariant under the action of the Coxeter group $M$:  $d(  \sigma  g,\tau g)=d(\sigma,\tau)$ for all $g\in M$ and $\sigma,\tau\in T$.
\end{Remark}

Let us now denote  by $J_\sigma$ , for $\sigma\in T$,  the $K$ or $L$ function associated with $\sigma$---that is, $J_\sigma(\vec{x})= L _\sigma (\vec{x})$ if $\sigma\in G_L\backslash M$, and $J_\sigma(\vec{x})= K_\sigma( \vec{x})$ if $\sigma\in G_K\backslash M$.  Let us also write
$$T^{(3)}=\{\hbox{\rm three-element subsets }\mathcal{S}\subset T\}.$$

 We have the following main theorem.

\begin{Theorem}\label{bigthm}
\begin{enumerate}
\item[\rm(a)] The action of $M$ by right multiplication (elementwise) on $T^{(3)}$ partitions this set  into eighteen orbits. More specifically, there are:

\begin{enumerate}\item[\rm(i)]  two orbits  each of whose elements $\mathcal{S}$ is a set containing three elements of $G_L\backslash M$; 

\item[\rm(ii)]  four  orbits each of whose elements $\mathcal{S}$ is a set  containing one element of $G_K\backslash M$ and two elements of $G_L\backslash M$; 

\item[\rm(iii)]  seven   orbits each of whose elements $\mathcal{S}$ is a set containing one element of $G_L\backslash M$ and two elements of $G_K\backslash M$; 

\item[\rm(iv)]   five orbits   each of whose elements $\mathcal{S}$ is a set containing three elements of  $G_K\backslash M$. 
\end{enumerate}

\item[\rm(b)] Let $\{\sigma,\tau,\mu\}\in T^{(3)}.$ Then there is a relation of the form
\begin{equation} \label{relngeneral}
\gamma_1 J_{\sigma} (\vec{x}) + \gamma_2 J_{\tau} (\vec{x}) + \gamma_3 J_{\mu} (\vec{x}) = 0,
\end{equation}
where each of the coefficients $ \gamma_1, \gamma_2, \gamma_3$ is a product of:
\begin{enumerate}\item[\rm(i)] at most  one function of the form $S(x)$, where $x$ is a coordinate of $g \vec{x}$ for some $g\in M$, and

\item[\rm(ii)] at most four functions of the form $A(g\vec{x})$ ($g\in M$), and 

\item[\rm (iii)] at most one function of the form $B_k(g\vec{x})$ ($k=1,2,3,4$, or $5$)  or $C(g\vec{x})$ ($g\in M$).
\end{enumerate} 
\item[\rm(c)] If $\{\sigma,\tau,\mu\} $ and $\{\sigma',\tau',\mu'\}   $ are in the same orbit under the action of $M$ on $T^{(3)}$ described above, then a three-term relation among
$J_{\sigma}$, $J_{\tau}$, and $J_{\mu}$ can be transformed into one among $J_{\sigma'}$, $J_{\tau'}$, and $J_{\mu'}$ by the
application of a single change of variable
$$\vec{x} \mapsto \rho \vec{x} \hspace{1in} (\rho \in M)$$
to all elements (including the coefficients) of the first relation.
 
\item[\rm(d)] For any $\ell \in \{1,2,3\}$, let $\{j,k\}=\{1,2,3\} \setminus \{\ell\}$.  Then, in a relation  \eqref{relngeneral}, each coefficient $\gamma_\ell$ has width $2^{d(\mu_j,\mu_k)/2-1}$.

\item[\rm(e)]   In any three-term relation  (\ref{relngeneral}), all coefficients have length
$$2(a+b+c)-6,$$where $abc$ is the type of $\{\sigma,\tau,\mu\}$.

\end{enumerate}

 \end{Theorem}

Part (a) of the above theorem has already been demonstrated---see Sections 3 and 4 above.  To prove the other parts of the theorem, it will suffice to produce, for each of the eighteen orbits described above, a relation  (\ref{relngeneral}) for one particular representative $\{\sigma,\tau,\mu\}$ of that orbit, such that each of the eighteen relations produced satisfies the conclusions of parts (d) and (e) of the theorem.

We will present the required relations in Propositions \ref{firstprop}--\ref{lastprop} below.  For the proofs of these, we will first need the following trigonometric identities.

\begin{Lemma} \label{triglemma}  \begin{enumerate}\item[\rm (a)] For $p,q,r,s\in\mathbb{C}$, we have
 $$S( p-q,r-s)-S(p-r,q-s)=-S(q-r,p-s).$$

\item[\rm (b)]     For $(a,b,c,d,e,f,g)^T\in V$, we have
\begin{align*} & S(g-a,f-b,f-c,f-d)-S(f-a,g-b,g-c,g-d)  \\=\,& S(f-g)\bigl[{S(b-a,c-a,d-a)+S(e-2a,f-a,g-a)}\bigr].\end{align*}
\end{enumerate}\end{Lemma}

\begin{proof}Part (a) may be proved directly, by expanding
 $$S( p-q,r-s)-S(p-r,q-s)+S(q-r,p-s)$$into sines and cosines of $p,q,r$, and $s$ alone, and observing that the resulting summands cancel pairwise.

To prove part (b) we note that, by part (a),  
\begin{align*}  &S( b-a,c-a,d-a)+S( e-2a,f-a,g-a) \\=\,& S( b-a)\bigl[S(g-c,g-d)-S(g-a,g+a-c-d)\bigr] \\+\,& S( g-a)\bigl[S(a+b-e,b-f)-S(b-a,c+d-a-g)\bigr]\\=\,& S( b-a,g-c,g-d)+S( a+b-e,b-f,g-a),\end{align*}so again by part (a),
\begin{align*} & S(g-a,f-b,f-c,f-d)-S(f-a,g-b,g-c,g-d) \\-\,&S(f-g)\bigl[{S(b-a,c-a,d-a)+S(e-2a,f-a,g-a)}\bigr]\\=\, & S(g-a,f-b)\bigl[S(f-c,f-d)+S(f-g,a+b-e)\bigr]\\-\,& S(g-c,g-d) \bigl[S(f-a,g-b) +S(f-g ,b-a)\bigr]\bigr]
\\=\,& S(g-a,f-b) S(g-c,g-d)- S(g-c,g-d) S(f-b,g-a) \\=\,&0.\end{align*} \end{proof}

\section{Explicit statement and derivation of the three-term relations}

We now proceed with the derivations of our relations.

\subsection{Type $222$ relations} 

There are four relations of Type $222$: the Orbit $1$ $(L,K,K)$ relation, the Orbit $1$ $(K,L,L)$ relation, the Hamming type $222$ $(K,K,K)$ relation, and the $L$-coherent $(L,L,L)$ relation.  In any such  relation, the coefficient of any $K$ or $L$ function has length $2(2+2+2)-6=6$.

\begin{Proposition}\label{firstprop}
We have the Orbit $1$ $(L,K,K)$ relation \begin{align}
\label{Orbit1KKL}
&S(a) A({b,1+b-e,1+b-f,1+b-g })\ser{K}{p_0} \nr-\,&S(b) A({a,1+a-e,1+a-f,1+a-g })\ser{K}{p_1}\nr
-\,&{S (b-a)  A(a,b,c,d)   L_{3}(\vec{x}) } 
=0.
\end{align}
\end{Proposition}

\begin{proof} By the definitions of $\ser{K}{p_0}$ and $\ser{K}{p_1}$, the left hand side of (\ref{Orbit1KKL}) equals\begin{align*}&\frac{S(a)}{\gg{b}\gg{1+b-e}\gg{1+b-f}\gg{1+b-g}}\\\times\,&\frac{1}{S(a)\gg{a}\gg{b}\gg{c}\gg{d}\gg{a}\gg{1+a-e}\gg{1+a-f}\gg{1+a-g}}\nonumber\\\times \,&
\biggl({_4}F_3^* \biggl[\genfrac{} {}{0pt}{}{a,b,c,d;}{ e,f,g;}1\biggr] +{_4}F_3^* \biggl[\genfrac{} {}{0pt}{}{a,1+a-e,1+a-f,1+a-g;}{1+a-b,1+a-c,1+a-d;}1\biggr]\biggr)\\-\,&\frac{S(b)}{\gg{a}\gg{1+a-e}\gg{1+a-f}\gg{1+a-g}}\\\times\,&\frac{1}{S(b)\gg{a}\gg{b}\gg{c}\gg{d}\gg{b}\gg{1+b-e}\gg{1+b-f}\gg{1+b-g}}\nonumber\\\times \,&
\biggl({_4}F_3^* \biggl[\genfrac{} {}{0pt}{}{b,a,c,d;}{ e,f,g;}1\biggr] +{_4}F_3^* \biggl[\genfrac{} {}{0pt}{}{b,1+b-e,1+b-f,1+b-g;}{1+b-a,1+b-c,1+b-d;}1\biggr]\biggr)\\=\,&\frac{1}{\biggl[ \genfrac{} {}{0pt}{}{\ds\gg{a}\gg{b}\gg{c}\gg{d}\gg{a}\gg{1+a-e}\gg{1+a-f}}{ \ds\times\gg{1+a-g}\gg{b}\gg{1+b-e}\gg{1+b-f}\gg{1+b-g}}\biggr]}\nonumber\\\times \,&
\biggl({_4}F_3^* \biggl[\genfrac{} {}{0pt}{}{a,1+a-e,1+a-f,1+a-g;}{ 1+a-b,1+a-c,1+a-d;}1\biggr]\\-\,&{_4}F_3^* \biggl[\genfrac{} {}{0pt}{}{b,1+b-e,1+b-f,1+b-g;}{ 1+b-a,1+b-c,1+b-d;}1\biggr]\biggr) \\=\,&S(b-a)A(a,b,c,d)\ser{L}{3}.\end{align*}\end{proof}

\begin{Proposition}
We have the Orbit $1$ $(K,L,L)$ relation
\begin{align}\label{Orbit1KLL}
 &S(f-e)A({ 1 - a,e-a, f-a, g-a })  \ser{K}{p_0} \nr+\,& S(f-a)A({e - a, e - b, e - c, e - d })\ser{L}{6} \nr-\,&  S(e-a)A({f - a, f - b, f - c, f - d })\ser{L}{5}
=0. 
\end{align} 
\end{Proposition} 
\begin{proof}

Consider the transformations$$\vec{x}\to({a,b,e-c,e-d; e, 1+a+b-f,1+a+b-g})^T$$and$$\vec{x}\to({a,b,f-c,f-d; f,1+a+b-e,1+a+b-g})^T.$$Application of the first of these to our Orbit $1$ $(L,K,K)$ relation (\ref{Orbit1KKL}) yields a relation among $\ser{K}{p_0}$, $\ser{K}{p_1}$, and  $\ser{L}{6}$, while application of the second to  (\ref{Orbit1KKL}) yields a relation among
$\ser{K}{p_0}$, $\ser{K}{p_1}$, and $\ser{L}{5}$. Eliminating $\ser{K}{p_1}$ from this pair of newly formed relations, and applying the identity $\gg{s}\gg{1-s}=\pi/\sin\pi s$, we get\begin{align}
\label{Orbit1KLL_2}
&  \Bb{1-a}{e-a}{f-a}{g-a}{\bigl[S(f-a,e-b)-S(e-a,f-b)} \bigr]{  K_{p_0}(\vec{x}) }\nr-\,&{S( f-a,b-a)  \Bb{e-a}{e-b}{e-c}{e-d} L_{6}(\vec{x}) }\nr+\,&S(  e-a,b-a) \Bb{f-a}{f-b}{f-c}{f-d}  L_{5}(\vec{x}) 
=0.
\end{align}But Lemma(\ref{triglemma})(a) tells us that $$S(f-a,e-b)-S(e-a,f-b)=S(a-b,f-e);$$dividing (\ref{Orbit1KLL_2}) through by $S(  a-b)$ therefore yields (\ref{Orbit1KLL}).
\end{proof}

\begin{Proposition}
We have the  Hamming type $222$ $(K,K,K)$ relation
\begin{align}\label{KKK222}&S(c-b)A({\ds 1 - a,e-a, f-a,  g-a }) \ser{K}{p_0}
\nr+\,&S(a-c)A({1 - b,e-b, f-b,  g-b })\ser{K}{p_1}\nr+\,&S(b-a)A({1 - c,e-c, f-c,  g-c })\ser{K}{p_2}=0. \end{align}
\end{Proposition}

\begin{proof}  This is Proposition 7.3 in \cite{FGS} (under the renormalization of $K(\vec{x})$ described in  Section 2---cf. Remark \ref{R210}---above).  \end{proof}

\begin{Proposition}
We have the $L$-coherent $(L,L,L)$ relation 
\begin{align}\label{coherLLL}
&S(f-g)A({e - a, e - b, e - c, e - d })\ser{L}{6}  \nr+\,& S(g-e)A({f - a, f - b, f - c, f - d })\ser{L}{5} \nr
+\,&S(e-f)A({g - a, g - b, g - c, g - d})  \ser{L}{4}
=0. 
\end{align}\end{Proposition}
 
\begin{proof}  This is Proposition 7.2 in \cite{M1}.\end{proof}

\subsection{Type $224$ relations} 

There are five relations of Type $224$: the Orbits 3 and 4 $(L,K,K)$ relations, the $L$-incoherent   $(L,L,L)$ relation, the Hamming type $224$ $(K,K,K)$ relation, and the Orbit $2$ $(K,L,L)$ relation. 
In any such relation, the coefficient of any $K$ or $L$ function has length $2(2+2+4)-6=10$.

\begin{Proposition}  We have the Orbit $3$ $(L,K,K)$ relation \begin{align}\label{Orbit3KKL}
& S(c)A(1 - a, e - a, f - a, g - a)  A({b, 1 + b - e, 1 + b - f, 1 + b - g} )  \ser{K}{p_0} \nr- \,&S(b-a)A  ({a,b,c,d } )A(1 - c, e - c, f - c, g - c )\ser{L}{\overline{1}} \nr-\, &B_1(\vec{x})   K_{p_1}(\vec{x})  
=0. 
\end{align}
\end{Proposition}

\begin{proof}Applying, to our Orbit $1$ $(K,L,L)$ relation (\ref{Orbit1KLL}), the change of variable $$\vec{x}\to({b, 1+b-e,1+b-f,1+b-g,1+b-a,1+b-c,1+b-d}),$$ we obtain a relation among $K_{p_1}(\vec{x})$,  $L_3(\vec{x})$, and $L_{\overline{1}}(\vec{x})$. We form a linear combination of this relation with  our Orbit $1$ $(L,K,K)$ relation (\ref{Orbit1KKL}), to eliminate $L_3(\vec{x})$. The result is \begin{align}\label{Orbit3KKL_2}&S(a,c )A({1-a,e-a,f-a,g-a }) \nr\times\,& A({b,1+b-e,1+b-f,1+b-g })\ser{K}{p_0} \nr-\,&  S(a,b-a)  A(a,b,c,d) A({1-c,e-c,f-c,g-c})\ser{L}{\overline1} \nr+\,
 & \bigl[S(a-c,b-a)  A(a,b,c,d) A({ 1 - a,1-b, 1-c, 1-d })   \nr -\,&S(b,c )A({1-a,e-a,f-a,g-a }) A({a,1+a-e,1+a-f,1+a-g })\bigr]\nr\times\,&\ser{K}{p_1}
=0. 
\end{align} Because $\gg{s}\gg{1-s}=\pi/\sin\pi s$, the quantity in square brackets,  in (\ref{Orbit3KKL_2}), equals \begin{align*}   &-\pi^{-4}\bigl[S(a,b,c,d,b-a,c-a)-S(a,b,c,e-a,f-a,g-a)\bigr]\\=\,&-S(a)B_1(\vec{x}),\end{align*}so dividing (\ref{Orbit3KKL_2}) through by $S(a)$  yields (\ref{Orbit3KKL}).
\end{proof}

\begin{Proposition}  We have the Orbit $4$ $(L,K,K)$ relation \begin{align}
\label{Orbit4KKL}
&S(f-a)A(1 - a, e - a, f - a, 
  g - a)\nr\times \,&A({1 + a - g, 1 + b - g, 1 + c - g, 1 + d - g }) \ser{K}{p_0} \nr-\,&   S(g-a)A({f-a,f-b,f-c,f-d })   \nr\times\,& A(a, 1 + a - e, 1 + a - f, 1 + a - g) \ser{K}{n_4} \nr-\,&  B_2(\vec{x})  \ser{L}{4} 
=0. 
\end{align}
\end{Proposition}

\begin{proof}
By applying, to our Orbit $1$ $(K,L,L)$ relation (\ref{Orbit1KLL}),  the interchange of $e$ and $g$ and then, separately,    the change of variable$$ \vec{x}\to({e-a,e-b,e-c,e-d,1+e-f,1+e-g,e})^T,$$we obtain a relation among $K_{p_0}(\vec{x})$, $L_4(\vec{x})$, and $L_5(\vec{x})$, and another among $K_{n_4}(\vec{x})$, $L_4(\vec{x})$, and $L_5(\vec{x})$.  From these latter two three-term relations, we form a linear combination to eliminate $L_5(\vec{x})$.  We get
\begin{align}
\label{Orbit4KKL_2}& S(f-a,f-g)A({ 1 - a,e-a, f-a, g-a })\nr\times\,
 &  A({1+a-g,1+b-g,1+c-g,1+d-g }) \ser{K}{p_0} \nr-\,&S(g-a,f-g)A({f - a, f - b, f - c, f - d })\nr\times\,& A({ 1+a-e,1+a-f, 1+a-g,a})  \ser{K}{n_4} \nr-\bigl[&S(g-a,g-a)A({f - a, f - b, f - c, f - d })\nr\times\, &  A({1+a-f,1+b-f,1+c-f,1+d-f})\nr-\,& S(f-a,f-a)A({g - a, g - b,g- c, g - d }) \nr\times\,&  A({1+a-g,1+b-g,1+c-g,1+d-g })\bigr]\ser{L}{4}  
=0. \end{align}
But, because $\gg{s}\gg{1-s}=\pi/\sin\pi s$, and by Lemma \ref{triglemma}(b), the above term in square brackets equals\begin{align}&\pi^{-4}S(f-a,g-a)\nr\times\,&[S(g-a,f-b,f-c,f-d)-S(f-a,g-b,g-c,g-d)]\nr=\,&\pi^{-4}S(f-a,g-a,f-g)[S(b-a,c-a,d-a)+S(e-2a,f-a,g-a)]\nr=\,&S(f-g)B_2(\vec{x}).\end{align} So dividing (\ref{Orbit4KKL_2}) through by $S(f-g)$ yields (\ref{Orbit4KKL}).  \end{proof}

\begin{Proposition}  We have the $L$-incoherent $(L,L,L)$ relation
\begin{align}\label{incoherLLL}
&S(g) \Bb{1+a-f}{1+b-f}{1+c-f}{1+d-f}\nr\times\,&A ({e - a, e - b, e - c, e - d})\ser{L}{6}    \nr +\,&  S(e-f)\Bb{a}{b}{c}{d} A({g - a, g - b, g - c, g - d})\ser{L}{\overline6}
\nr-\, &   B_3(\vec{x}) \ser{L}{5}
=0. 
\end{align}
 \end{Proposition}

\begin{proof} Multiplying equation (7.7) in \cite{M1}  through by\begin{align*}&\frac{S(e-f,g)}{S(f-g)}\\\times\,&B(1+a-f,1+b-f,1+c-f,1+d-f)B(g-a,g-b,g-c,g-d),\end{align*}and applying the identity $\gg{s}\gg{1-s}=\pi/\sin\pi s$, yields

\begin{align*}
& S(g)\Bb{1+a-f}{1+b-f}{1+c-f}{1+d-f} \\\times\,&\Bb{e-a}{e-b}{e-c}{e-d} \ser{L}{6}  \nr- \,&\frac{1}{\pi^4 A(f-g)} \biggl[\genfrac{} {}{0pt}{}{ A( g,g-e,f-a,f-b,f-c,f-d)}{ -A( f,e-f,g-a,g-b,g-c,g-d) } \biggr]\ser{L}{5}\nr+\,&S(e-f)\Bb{a}{b}{c}{d} \Bb{g-a}{g-b}{g-c}{g-d}  \ser{L}{\overline6}
=0, 
\end{align*}which is precisely (\ref{incoherLLL}).\end{proof}

\begin{Proposition}  We have the Orbit $2$ $(K,L,L)$ relation \begin{align}\label{Orbit2KLL}
&S(g) \Bb{1+a-f}{1+b-f}{1+c-f}{1+d-f}\nr\times\,&  A({ 1 - a,e-a, f-a, g-a })  \ser{K}{p_0}\nr+\,&S(f-a)  \Bb{a}{b}{c}{d} A({g - a, g - b, g - c, g - d})\ser{L}{\overline6}   \nr-\,&B_4(\vec{x}) \ser{L}{5}=0. 
\end{align}\end{Proposition}

\begin{proof}We take a linear combination of the Orbit $1$ $(K,L,L)$ relation (\ref{Orbit1KLL}) with the $L$-incoherent $(L,L,L)$ relation (\ref{incoherLLL}), to eliminate  $\ser{L}{6}$.  We get
\begin{align}\label{Orbit2KLL_2}&S(g,f-e) \Bb{1+a-f}{1+b-f}{1+c-f}{1+d-f}\nr\times\,&  A({ 1 - a,e-a, f-a, g-a })  \ser{K}{p_0}\nr-\,&S(f-a,e-f)\Bb{a}{b}{c}{d} A({g - a, g - b, g - c, g - d})\ser{L}{\overline6}   \nr-\,& \bigl[S(g,e-a) \Bb{1+a-f}{1+b-f}{1+c-f}{1+d-f}\nr\times\,&  A({f - a, f - b, f - c, f - d }) 
\nr-\, &  S(f-a) B_3(\vec{x})\bigr]
 \ser{L}{5}=0. 
\end{align}We observe that the quantity in square brackets, in (\ref{Orbit2KLL_2}), equals
\begin{align*}&\frac{1}{\pi^4S(f-g)}\left[ \begin{matrix}S(g,f-g,e-a,f-a,f-b,f-c,f-d)\\- S(f,f-e,f-a,g-a,g-b,g-c,g-d) \\+ S(g,g-e,f-a,f-a,f-b,f-c,f-d) \end{matrix} \right]\\=\,&\frac{1}{\pi^4S(f-g)} \left[ \begin{matrix}S(g,f-a,f-b,f-c,f-d)\\\times\bigl[S(f-g,e-a)+S(g-e,f-a)\bigr]\\- S(f,f-e,f-a,g-a,g-b,g-c,g-d)  \end{matrix}\right]
\\=\,&\frac{S(f-e,f-a,g-a)\bigl[ S(g, f-b,f-c,f-d)-S(f ,g-b,g-c,g-d)  \bigr]}{\pi^4S(f-g)}
\\=\,&\frac{S(f-e,f-a,g-a)\bigl[ S(b,c,d)+S(e-a,f,g)  \bigr]}{\pi^4 }=S(f-e)B_4(\vec{x}),\end{align*}where we have used parts (a) and (b) of Lemma (\ref{triglemma}).  Dividing equation (\ref{Orbit2KLL_2}) through by $S(f-e)$ then yields (\ref{Orbit2KLL}).\end{proof}

\begin{Proposition}  We have the  Hamming type $224$ $(K,K,K)$ relation\begin{align}\label{KKK224} & S(e-a-b) \Bb{1-a}{e - a}{f - a}{g-a}  \nr\times\, &A ({b,{1+b-e}, {1+b-f} , 1 + b - g})\ser{K}{p_0} \nr+\,& S(b-a)\Bb{1+a-e}{1+b-e}{g-c}{g-d}\nr \times\,&A ({a, b, f - c, f - d })K_{n_4}(\vec{x})  \nr-\,& B_5(\vec{x})K_{p_1}(\vec{x})
=0. \end{align}\end{Proposition}

\begin{proof}  If we divide equation (7.4) in \cite{FGS} through by $\pi^3\Gamma(b)$, and recall that the function $K(\vec{x})$ defined there is $\pi/\gg{1-a}$ times the function $K(\vec{x})$ appearing here, then we get (in terms of the present normalization of $K(\vec{x})$):

\begin{align}\label{KKK224_2}& S(e-a-b)\Bb{1-a}{e-a}{f-a}{g-a} \nr\times\,& \Bb{b}{1+b-e}{1+b-f}{1+b-g}K_{p_0}(\vec{x})\nr+\,&  \frac{S(b,e-b)}{ \pi^4 S(c-b)}\biggl[\genfrac{} {}{0pt}{}{\sin\pi(a-c)\sin\pi(e-a-b)\sin\pi(f-b)\sin\pi(g-b)}{-
\sin\pi(a-b)\sin\pi(e-a-c)\sin\pi(f-c)\sin\pi(g-c) }\biggr] \nr\cdot\,&K_{p_1}(\vec{x})
\nr+\,& S(b-a) \Bb{1+a-e}{1+b-e} {g-c}{g-d}\Bb{a}{b}{f-c}{f-d}K_{n_4}(\vec{x})\nr =
\,& 0.\end{align}But by Lemma \ref{triglemma}(b), \begin{align*}&  \frac{S(b,e-b)}{ \pi^4 S(c-b)}\biggl[\genfrac{} {}{0pt}{}{\sin\pi(a-c)\sin\pi(e-a-b)\sin\pi(f-b)\sin\pi(g-b)}{-
\sin\pi(a-b)\sin\pi(e-a-c)\sin\pi(f-c)\sin\pi(g-c) }\biggr]\\=-\,&   \frac{S(b,e-b)}{ \pi^4 }\bigl[S(b-a,c-a,d-a)+S(e-2a,f-a,g-a)\bigr] =-B_5(\vec{x}), \end{align*}so (\ref{KKK224_2}) yields (\ref{KKK224}).
\end{proof}
 
\subsection{Type $244$ relations}  

There are five relations of Type $244$: the Orbits 3 and 4 $(K,L,L)$ relations, the Orbits 2 and 6 $(L,K,K)$ relations, and the Hamming type $244$ $(K,K,K)$ relation. In any such relation, the coefficient of any $K$ or $L$ function  has length $2(2+4+4)-6=14$.

\begin{Proposition}  We have the Orbit $3$  $(K,L,L)$ relation   \begin{align}\label{Orbit3KLL}
&  A(1 - a, e-a,f-a,g-a)B_3(\vec{x})   K_{p_0}(\vec{x})   \nr+\,&S(e-a)  A(f-a,f-b,f-c ,f-d) \nr\times\,&A(g-a, g-b,g-c,g-d)A({a, b, c, d}) \ser{L}{\overline6} \nr -\,&  A(e-a,e-b,e-c ,e-d)  B_4(\vec{x})\ser{L}{6}
=0. 
\end{align}\end{Proposition}

\begin{proof} We form a linear combination of our Orbit $1$ $(K,L,L)$ relation (\ref{Orbit1KLL}) and our $L$-incoherent $(L,L,L)$ relation (\ref{incoherLLL}), to eliminate $\ser{L}{5}$.   We get
\begin{align}\label{Orbit3KLL_2}
 & S(f-e)A({ 1 - a,e-a, f-a, g-a }) B_3(\vec{x}) \ser{K}{p_0}\nr-\,& S(e-a,e-f)A({f - a, f - b, f - c, f - d })   \nr \times\,&   \Bb{a}{b}{c}{d} A({g - a, g - b, g - c, g - d})\ser{L}{\overline6}
 \nr+\,&A({e - a, e - b, e - c, e - d })\nr\times\,&\bigl[ S(f-a)B_3(\vec{x}) -S(e-a,g)A({f - a, f - b, f - c, f - d })
\nr\times\,& 
  \Bb{1+a-f}{1+b-f}{1+c-f}{1+d-f} \bigr]  \ser{L}{6}  =0. 
\end{align}But the term in square brackets, in (\ref{Orbit3KLL_2}), equals 
\begin{align*}&\frac{1 }{\pi^4S(f-g)}    \left[\begin{matrix} S(f-a,f,f-e,g-a,g-b,g-c,g-d) \\ -S(f-a,g,g-e,f-a,f-b,f-c,f-d)\\ -S(f-g,e-a,g,f-a,f-b,f-c,f-d) \end{matrix} \right]
\\=\,&\frac{1 }{\pi^4S(f-g)}    \left[\begin{matrix} S(f-a,f,f-e,g-a,g-b,g-c,g-d) \\ -S( g, f-a,f-b,f-c,f-d)\\ \times [S(f-a,g-e  )+S(f-g,e-a  ) \bigr]\end{matrix} \right]
\\=\,&\frac{1 }{\pi^4S(f-g)}    \left[\begin{matrix} S(f-a,f,f-e,g-a,g-b,g-c,g-d) \\ -S(f-e,g-a, g, f-a,f-b,f-c,f-d) \end{matrix} \right]
\\=\,&\frac{S(f-e,f-a,g-a) }{\pi^4S(f-g)}    \left[\begin{matrix} S( f ,g-b,g-c,g-d) \\ -S( g,f-b,f-c,f-d) \end{matrix} \right]
\\=\,&-\frac{S(f-e,f-a,g-a) }{\pi^4 }    \left[\begin{matrix} S( b,c,d) \\ +S( e-a,f,g) \end{matrix} \right]=-S(f-e)B_4(\vec{x}),\end{align*}where we have used parts (a) and (b) of Lemma \ref{triglemma}.  So dividing (\ref{Orbit3KLL_2}) through by $S(f-e)$ yields (\ref{Orbit3KLL}).
\end{proof}

\begin{Proposition}  We have the Orbit $4$ $(K,L,L)$ relation\begin{align}\label{Orbit4KLL}
&S(d-c)A(1 - a, e - a, f - a, g - a) \nr\times\,&A(1-a,1-b,1-c,1-d)A({b, 1 + b - e, 1 + b - f, 1 + b - g} )\ser{K}{p_0} \nr- \,& A(1 - c, e - c, f - c, g - c )B_1(a,b,d,c,e,f,g) \ser{L}{\overline{1}}   \nr+ \,& A(1 - d, e - d, f - d, g - d )B_1(\vec{x})\ser{L}{\overline{2}}    
=0. 
\end{align}\end{Proposition}

\begin{proof} From our Orbit $3$ $(L,K,K)$ relation (\ref{Orbit3KKL}) and its image under the transposition of $c$ and $d$, we form a linear combination to eliminate $\ser{K}{p_1}$.  We get \begin{align}\label{Orbit4KLL_2}
&A(1 - a, e - a, f - a, g - a) A({b, 1 + b - e, 1 + b - f, 1 + b - g} )\nr\times\,&\bigl[ S(c) B_1(a,b,d,c,e,f,g)  
-S(d) B_1(\vec{x})\bigr] \ser{K}{p_0} \nr- \,&S(b-a)A  ({a,b,c,d } )A(1 - c, e - c, f - c, g - c )B_1(a,b,d,c,e,f,g) \ser{L}{\overline{1}}   \nr+ \,&S(b-a)A  ({a,b,c,d } )A(1 - d, e - d, f - d, g - d )B_1(\vec{x})\ser{L}{\overline{2}}    
=0. 
\end{align}
But the quantity in square brackets, in (\ref{Orbit4KLL_2}), equals 
\begin{align*}&\frac{S(b,c,d)} {\pi^{4}}\biggl[\genfrac{} {}{0pt}{}{S( b-a,c,d-a)  +S(e-a,f-a,g-a)}{ -S( b-a,d,c-a)  -S(e-a,f-a,g-a)}\biggr] \\=\,&\frac{S(b,c,d,b-a)\bigl[S(d-a,c)-S(  c-a,d)\bigr]  } {\pi^{4}}=\frac{S(a,b,c,d,b-a,d-c)  } {\pi^{4}} \end{align*}by Lemma (\ref{triglemma}), so dividing (\ref{Orbit4KLL_2}) through by $S(b-a)A(a,b,c,d)$ yields (\ref{Orbit4KLL}).
\end{proof}

\begin{Proposition}  We have the Orbit $2$ $(L,K,K)$ relation   
  \begin{align}\label{Orbit2KKL}
&   A({ 1 - a,e-a, f-a, g-a }) B_4(b,a,c,d,e,f,g)  \ser{K}{p_0}  \nr-\,&    A({ 1 - b,e-b, f-b, g-b }) B_4(\vec{x}) \ser{K}{p_1}\nr-\, & S(b-a)A({f - a, f - b, f - c, f - d}) A({g - a, g - b, g - c, g - d})\nr\times\,&\Bb{a}{b}{c}{d}  \ser{L}{\overline6}  =0. 
\end{align}\end{Proposition}

\begin{proof}Transposing $a$ and $b$ in our Orbit $2$ $(K,L,L)$ relation (\ref{Orbit2KLL}) yields a relation among $\ser{K}{p_1}$, $\ser{L}{\overline6}$, and $\ser{L}{5}$.  An appropriate combination of this relation with (\ref{Orbit2KLL}) effects the elimination of $\ser{L}{5}$, thus:
   \begin{align}\label{Orbit2KKL_2}
& S(g) \Bb{1+a-f}{1+b-f}{1+c-f}{1+d-f}\nr\times\,&  A({ 1 - a,e-a, f-a, g-a }) B_4(b,a,c,d,e,f,g)  \ser{K}{p_0}  \nr-\,& S(g) \Bb{1+a-f}{1+b-f}{1+c-f}{1+d-f}\nr\times\,&  A({ 1 - b,e-b, f-b, g-b }) B_4(\vec{x}) \ser{K}{p_1}\nr-\, & \Bb{a}{b}{c}{d} A({g - a, g - b, g - c, g - d})\nr\times\,&\bigl[S(f-b)B_4(\vec{x})+S(f-a)  B_4(b,a,c,d,e,f,g) \bigr]\ser{L}{\overline6}  =0. 
\end{align}But the quantity in square brackets, in  (\ref{Orbit2KKL_2}), equals
\begin{align}\label{Orbit2KKL_3}&\frac{1}{\pi^4}\biggl[\genfrac{} {}{0pt}{}{S(f-b, f-a,g-a)\bigl[{S(b,c,d) +S(e-a,f,g)}\bigr]}{-S(f-a,f-b,g-b)\bigl[{S(a,c,d) +S(e-b,f,g)}\bigr] }\biggr]
\nr=\,&\frac{S(f-a,f-b)}{\pi^4}\biggl[\genfrac{} {}{0pt}{}{S(c,d) \bigl[{S(g-a,b)-S(g-b,a)}\bigr]}{+S(f,g) \bigl[{S(g-a,e-a) +S(g-b,e-b)}\bigr] }\biggr]
\nr=\,&{\pi^{-4}}{S(f-a,f-b)}\bigl[{S(c,d,b-a,g)  +S(f,g,b-a,f-c-d)   }\bigr]
\nr=\,&{\pi^{-4}}{S(f-a,f-b,g,b-a)} \bigl[{S(c,d ) +S(f,f-c-d)   }\bigr]
\nr=\,&{\pi^{-4}}{S(f-a,f-b,f-c,f-d,g,b-a)} . \end{align}So dividing (\ref{Orbit2KKL_2}) through by $ S(g) \Bb{1+a-f}{1+b-f}{1+c-f}{1+d-f}$ yields (\ref {Orbit2KKL}).\end{proof} 

\begin{Proposition}  We have the Orbit $6$ $(L,K,K)$ relation  
 \begin{align}\label{Orbit6KKL}&S(a) A({1 + a - f, 1 + b - f, 1 + c - f, 1 + d - f }) \nr\times\,&  A(1 + a - g, 1 + b - g, 1 + c - g, 1 + d - g )A(1 - a, e - a, f - a, g - a) \nr\times\,& \ser{K}{p_0}  \nr
-&A(a,1+a-e,1+a-f,1+a-g) B_4({\vec{x}}) K_{n_4}(\vec{x})  \nr-\,& A(a,b,c,d)   B_2(\vec{x})  \ser{L}{\overline6}
=0. 
\end{align}\end{Proposition}

\begin{proof}Applying the change of variable$$\vec{x}\to (e-a,e-b,e-c,e-d,e,1+e-f,1+e-g)^T$$ to  our Orbit $1$ $(K,L,L)$ relation (\ref{Orbit1KLL}) yields a relation among $\ser{K}{n_4}$, $\ser{L}{\overline6}$,  and $\ser{L}{4}$.  We combine this with our Orbit $4$ $(L,K,K)$ relation (\ref{Orbit4KKL}) to eliminate $\ser{L}{4}$.   We get
\begin{align}\label{Orbit6KKL_2}& S(a,f-a )A({1+a-f,1+b-f,1+c-f,1+d-f})\nr\times \,&A({1 + a - g, 1 + b - g, 1 + c - g, 1 + d - g })A(1 - a, e - a, f - a, 
  g - a) \nr\times\,& \ser{K}{p_0} \nr
-\, &A({ a,1+a-e,1+a-f, 1+a-g })\bigl[ S(f)B_2(\vec{x})  \nr+\,&   S(a,g-a )A({f-a,f-b,f-c,f-d })    \nr\times\,&A({1+a-f,1+b-f,1+c-f,1+d-f})   \bigr]\ser{K}{n_4}  \nr-\,& S(f-a)A({a,b,c,d})B_2(\vec{x}) \ser{L}{\overline6} 
=0. 
\end{align} But the quantity in square brackets, in (\ref{Orbit6KKL_2}), equals\begin{align}&\frac{S(f-a,g-a)}{\pi^4}\biggl[\genfrac{} {}{0pt}{}{S(f)  \bigl[{S(b-a,c-a,d-a)+S(e-2a,f-a,g-a)}\bigr]}{+S(a,f-b,f-c,f-d)}\biggr]\nr=\,&\frac{S(f-a,g-a)}{\pi^4}\biggl[\genfrac{} {}{0pt}{}{ {S(f,b-a,c-a,d-a)+S(a, f-b,f-c,f-d)}\bigr]}{+S(f,e-2a,f-a,g-a)}\biggr]\nr=\,&\frac{S(f-a,g-a,f-a)}{\pi^4}\biggl[\genfrac{} {}{0pt}{}{S( b,c,d)+S(a,f,g-e)}{+S(f,e-2a, g-a)}\biggr]\nr=\,& \frac{S(f-a,g-a,f-a)}{\pi^4}\biggl[\genfrac{} {}{0pt}{}{S( b,c,d) }{+S(e-a,f,g)}\biggr]= {S(f-a)}B_4(\vec{x}),\end{align}where we have used parts (a) and (b) of Lemma \ref{triglemma}.  So dividing (\ref{Orbit6KKL_2}) through by $S(f-a)$ yields (\ref{Orbit6KKL}).
\end{proof}

\begin{Proposition}  We have the Hamming type $244$ $(K,K,K)$ relation\begin{align}\label{KKK244}& S(f-e)	 \Bb{1-a}{1-b }{1+c-g}{1+d-g}     \nr\times&\Bb{b }{1+b-e}{1+b-f} {1+b-g}A({ 1 - a,e - a, f - a, g - a })K_{p_0}(\vec{x}) 
 \nr-& \Bb{1+a-e}{1+b-e}{f-c}{f-d}  B_5({a,b,c,d, f,e,g})  K_{n_4}(\vec{x})
  \nr+\,& \Bb{1+a-f}{1+b-f}{e-c}{e-d}B_5(\vec{x})  K_{n_8}(\vec{x})= 0.  \end{align}\end{Proposition}
\begin{proof} Multiplying equation (7.5) in \cite{FGS} through by $S(b)\gg{a} \gg{g-c}\gg{g-d}/\pi^4$, and recalling that the function $K(\vec{x})$ defined there is $\pi/\gg{1-a}$ times the one defined here, we get (in terms of the present function $K(\vec{x})$)
\begin{align}&\label{KKK244_2} S(f-e)	 \Bb{1-a}{1-b }{1+c-g}{1+d-g}     \nr\times\,&\Bb{b }{1+b-e}{1+b-f} {1+b-g}A({ 1 - a,e - a, f - a, g - a })K_{p_0}(\vec{x}) \nr+\,&\Bb{1+a-e}{1+b-e}{f-c}{f-d}  \nr\times\, & {\ds\frac{S(b,f-b)}{S(c-b)} \biggl[\genfrac{} {}{0pt}{}{\ds\sin\pi(a-c)\sin\pi(f-a-b)\sin\pi(e-b)\sin\pi(g-b)}{\ds-
\sin\pi(a-b)\sin\pi(f-a-c)\sin\pi(e-c)\sin\pi(g-c) }\biggr] }\nr\,\times&K_{n_4}(\vec{x})
\nr-\,& \Bb{1+a-f}{1+b-f}{e-c}{e-d}\nr\times\,& {\ds\frac{S(b,e-b)}{S(c-b)}  \biggl[\genfrac{} {}{0pt}{}{\ds\sin\pi(a-c)\sin\pi(e-a-b)\sin\pi(f-b)\sin\pi(g-b)}{\ds-
\sin\pi(a-b)\sin\pi(e-a-c)\sin\pi(f-c)\sin\pi(g-c) }\biggr] } \nr\times\,&K_{n_8}(\vec{x})= 0. \end{align}Applying Lemma \ref{triglemma}(b) to the terms in square brackets, in (\ref{KKK244_2}), yields (\ref{KKK244}).
\end{proof}

\subsection{Type $246$ relations}  

There are two relations of Type $246$: the Orbit $7$ $(L,K,K)$ relation and the Hamming type $246$ $(K,K,K)$ relation.   In any such  relation, the coefficient of any $K$ or $L$ function  has length $2(2+4+6)-6=18$.

\begin{Proposition}  We have the Orbit $7$ $(L,K,K)$  relation\begin{align} \label{Orbit7KKL}
&  A(1+a-e, 1 + b - e, 1 + c-e, 1 + d-e)  A(1 - a, e-a,f-a,g-a) \nr\times\,&
B_4(\vec{x}) K_{p_0}(\vec{x}) \nr-\, &S(e-a)A(f-a,f-b,f-c,f-d) A(g-a,g-b,g-c,g-d)    \nr\times\, &A(a, b,c,d)  A({a, 1 + a - e, 1 + a - f, 1 + a - g})K_{n_0}(\vec{x})  \nr-\,&C(\vec{x})\ser{L}{6}
=0. 
\end{align}
\end{Proposition}

\begin{proof}Applying the central involution$$\vec{x}\to(1-a,1-b,1-c,1-d,2-e,2-f,2-g)^T$$to our Orbit $2$ $(K,L,L)$ relation (\ref{Orbit2KLL}) yields a relation among $\ser{K}{n_0}$, $\ser{L}{6}$, and $\ser{L}{\overline5}$.  On the other hand, transposing $e$ and $f$ in (\ref{Orbit2KLL}) yields a relation among $\ser{K}{p_0}$, $\ser{L}{\overline5}$, and $\ser{L}{6}$. Of the two new relations thus formed, we fashion a linear combination   to eliminate $\ser{L}{\overline5}$.  We get
 \begin{align}\label{Orbit7KKL_2}
& S(g) \Bb{1+a-e}{1+b-e}{1+c-e}{1+d-e}\nr\times\,&  A({ 1 - a,e-a, f-a, g-a })B_4(\vec{x})  \ser{K}{p_0}  \nr-\,& S(g,e-a) \Bb{f-a}{f-b}{f-c}{f-d}A({g - a, g - b, g - c, g - d})\nr\times\,&  \Bb{a}{b}{c}{d} A({ a,1+a-e,1+a-f,1+a-g })   \ser{K}{n_0}\nr+\,&\bigl[S(e-a,f-a)\Bb{a}{b}{c}{d}  \Bb{1-a}{1-b}{1-c}{1-d}\nr\times\,&   A({g - a, g - b, g - c, g - d})A({1+a-g,1+b-g,1+c-g,1+d-g}) \nr-\,&B_4(\vec{x})B_4(a,b,c,d,f,e,g) \bigr]\ser{L}{ 6} =0. 
\end{align}But, by parts (a) and (b) of Lemma \ref{triglemma},  the quantity in square brackets, in (\ref{Orbit7KKL_2}), equals

\begin{align*}& {\pi^{-4}}\left[\begin{matrix}\pi^{-4}S(e-a,f-a,a,b,c,d,g-a,g-b,g-c,g-d)\\-\,  {S(e-a,g-a)}\bigl[{S(b,c,d) +S(e,f-a,g)}\bigr] B_4(\vec{x})   \end{matrix}\right] \nr=\,& \frac{S(e-a)}{\pi^{4}}\left[\begin{matrix}\pi^{-4}S( f-a,a,b,c,d,g-a,g-b,g-c,g-d)\\-S(b,c,d,g-a)B_4(\vec{x})   -S(e,f-a,g-a,g)B_4(\vec{x})   \end{matrix}\right] 
 \nr=\,& \frac{S(e-a)}{\pi^{4}}\left[\begin{matrix}\pi^{-4}S( f-a,a,b,c,d,g-a,g-b,g-c,g-d)\\-\pi^{-4} S(b,c,d,g-a,f-a,g-a)\bigl[{S(b,c,d) +S(e-a,f,g)}\bigr]  \\  -S(e,f-a,g-a,g)B_4(\vec{x})    \end{matrix}\right] 
 \nr=\,& -\frac{S(e-a)}{\pi^{4}}\left[\begin{matrix}\pi^{-4}S(b,c,d,f-a,g-a)\bigl[-S( a,g-b,g-c,g-d)\\+S(g-a,b,c,d) +S(e-a,f,g-a,g)\bigr]  \\  +S(e,f-a,g-a,g)B_4(\vec{x})    \end{matrix}\right] 
 \nr=\,& -\frac{S(e-a)}{\pi^{4}}\left[\begin{matrix}\pi^{-4}S(b,c,d,f-a,g-a)\bigl[  S(b-a,c-a,d-a,g)\\ -S(a,g-a,e+f-2a,g) +S(e-a,f,g-a,g)\bigr]  \\  +S(e,f-a,g-a,g)B_4(\vec{x})    \end{matrix}\right]
 \nr=\,& -\frac{S(e-a)}{\pi^{4}}\left[\begin{matrix}\pi^{-4}S(b,c,d,f-a,g-a)\\\times\bigl[  S(b-a,c-a,d-a,g) +S(e-2a,f-a,g-a,g)\bigr]  \\  +S(e,f-a,g-a,g)B_4(\vec{x})    \end{matrix}\right] 
 \nr=\,& -S(g){\pi^{-4}}S(e-a)\left[\begin{matrix} S(b,c,d)B_2(\vec{x})+S(e,f-a,g-a)B_4(\vec{x})    \end{matrix}\right] \nr=\,&-S(g)C(\vec{x}).\end{align*}
So dividing (\ref{Orbit7KKL_2}) through by $S(g)$ yields (\ref{Orbit7KKL}).
\end{proof} 

\begin{Proposition}  We have the Hamming type  $246$ $(K,K,K)$ relation\begin{align}\label{KKK246}
 &     A(1 - a, 1 - b, 1 - c, 1 - d) A(1 + a - e, 1 + b - e, 1 + c - e, 1 + d - e)\nr\times\,&B_2(\vec{x})\ser{K}{p_0} \nr+\,&S(e)  A(f - a, f - b, f - c, f - d)A( g - a, g - b, g - c, g - d) \nr\,\times\,&(A(a, 1 + a - e, 1 + a - f, 1 + a - g) )^2\ser{K}{n_{0}}     \nr-\,
& C(\vec{x})\ser{K}{p_4} 
=0. 
\end{align}\end{Proposition}
 \begin{proof}While this may be deduced from Proposition 7.7 in \cite{FGS}, a more direct, self-contained derivation may be given as follows.  Applying the change of variable$$\vec{x}\to (a,1+a-e,1+a-f,1+a-g,1+a-b,1+a-c,1+a-d)^T$$to our Orbit $1$ $(L,K,K)$ relation (\ref{Orbit1KKL}), we obtain a relation among $\ser{K}{p_0}$, $\ser{K}{p_4}$, and $\ser{L}{6}$.  Combining this is an appropriate fashion with our Orbit $7$ $(K,K,L)$  relation (\ref{Orbit7KKL}), we eliminate $\ser{L}{6}$ to get
\begin{align}
\label{KKK246_2}
&A({1+a-e,1+b-e,1+c-e,1+d-e})\bigl[S(a) C(\vec{x})  \nr-\,& S (e)    A(1 - a, e-a,f-a,g-a) \nr\times\,& A(a,1+a-e,1+a-f,1+a-g)
B_4(\vec{x}) \bigr]K_{p_0}(\vec{x}) \nr+\, &S(e-a,e)   A(f-a,f-b,f-c,f-d)  A(g-a,g-b,g-c,g-d)    \nr\times\, &A(a, b,c,d)  (A({a, 1 + a - e, 1 + a - f, 1 + a - g}))^2K_{n_0}(\vec{x}) \nr-\,&S(e-a) A({a,b,c,d })C(\vec{x})\ser{K}{p_4}
=0.
\end{align}But the quantity in square brackets, in (\ref{KKK246_2}), equals
\begin{align*}& S(a)\bigl[C(\vec{x}) -\pi^{-4}S(e,e-a,f-a,g-a)B_4(\vec{x})\bigr] \\=\,&  \pi^{-4}S(a,b,c,d,e-a)B_2(\vec{x}),
\end{align*}by the definition of $C(\vec{x})$.
So dividing (\ref{KKK246_2}) through by $S(e-a)A(a,b,c,d)$ yields (\ref{KKK246}).
\end{proof}

\subsection{Type $444$ relations}  

There are two relations of Type $444$: the Orbit $5$ $(L,K,K)$ relation and the Hamming type $444$ $(K,K,K)$ relation.   In any such  relation,  the coefficient of any $K$ or $L$ function has length $2(4+4+4)-6=18$.  (Note that, while the coefficients in a type $246$ relation have the same length as those in a type $444$ relation, the same cannot be said for the widths  of these coefficients.  In a relation of the type $246$, the three coefficients have widths $1,2$ and $4$ respectively, while, in a relation of type 444, all coefficients have width $2$.)

\begin{Proposition}  We have the  Orbit $5$ $(L,K,K)$  relation \begin{align}
 \label{Orbit5KKL}& A(1 - a, e-a, f-a, g-a)A(1 + a - f, 1 + b - f, 1 + c - f, 1 + d - f)  \nr\times\,& B_4(e - a, e - b, e - c, e - d, 1 + e - f, e, 1 + e - g) \ser{K}{p_0} \nr-\,  & A(g-a,g-b ,g-c,g-d )A(a, 1 +a - e, 1 + a - f, 1+a-g)  \nr\times\,  &  B_4(a,b,c,d,g,f,e)  \ser{K}{n_{4}}   \nr-\,&  A(a, b,c,d) A(e-a, e-b,e-c,e-d)   B_2({\vec{x}}) \ser{L}{\overline4} 
=0. 
\end{align}\end{Proposition}

\begin{proof} Transposing $f$ and $g$ in our Orbit $4$ $ (L,K,K)$ relation (\ref{Orbit4KKL}) yields a relation among $\ser{K}{p_0}$, $\ser{K}{n_4}$, and $\ser{L}{5}$, while transposing $e$ and $g$ in our Orbit $2$ $(K,L,L)$ relation (\ref{Orbit2KLL}) yields a relation among $\ser{K}{p_0}$, $\ser{L}{\overline4}$, and $\ser{L}{5}$. We form a linear combination of the two new relations thus formed  to eliminate $\ser{L}{5}$, to get
 \begin{align}
\label{Orbit5KKL_2}  &A(1 - a, e - a, f - a, 
  g - a)A({1 + a - f, 1 + b - f, 1 + c - f, 1 + d - f }) \nr\times\,&\bigl[S(g-a) B_4(a,b,c,d,g,f,e)-  S(e)   B_2(\vec{x}) \bigr]\ser{K}{p_0}\nr-\,&   S(f-a)A({g-a,g-b,g-c,g-d })   \nr\times\,& A(a, 1 + a - e, 1 + a - f, 1 + a - g) B_4(a,b,c,d,g,f,e)\ser{K}{n_4} \nr-\,&S(f-a)  \Bb{a}{b}{c}{d} A({e - a, e - b, e - e, g - d})B_2(\vec{x}) \ser{L}{\overline4}  
=0. 
\end{align}But the quantity in square brackets, in (\ref{Orbit5KKL_2}), equals \begin{align*}&{\pi^{-4}}{S(f-a,g-a)}\\\times\,&\biggl[\genfrac{} {}{0pt}{}{S(e-a )\bigl[{S(b,c,d) +S(e,f,g-a)}\bigr] }{- S(e)   \bigl[{S(b-a,c-a,d-a) +S(e-2a,f-a,g-a)}\bigr]}\biggr]\\=\,&{\pi^{-4}}{S(f-a,g-a)} \\\times\,&\biggl[\genfrac{} {}{0pt}{}{S(e,g-a )\bigl[{S(e-a,f)  +S(e-2a,f-a)}\bigr]}{ + S(e-a,b,c,d) -S(e,b-a,c-a,d-a)}\biggr]
\\=\,&{\pi^{-4}}{S(f-a,g-a)} \\\times\,&\biggl[\genfrac{} {}{0pt}{}{S(e,g-a,a,e+f-2a ) }{ + S(a)   \bigl[{S(e-b,e-c,e-d)+S(e,e-a,2a-f-g)}\bigr]}\biggr]
\\=\,&{\pi^{-4}}{S(a,f-a,g-a)} \\\times\,&\biggl[\genfrac{} {}{0pt}{}{S(e  )\bigl[S(g-a,e+f-2a)+S(e-a,2a-f-g)\bigr] }{ + S( e-b,e-c,e-d) }\biggr]\\=\,&{\pi^{-4}}{S(a,f-a,g-a)} \bigl[{S(e-b,e-c,e-d)  +S(e,a-f, e-g)   }\bigr]\\=\,&S(f-a)B_4(e - a, e - b, e - c, e - d, 1 + e - f, e, 1 + e - g).\end{align*}So dividing (\ref{Orbit5KKL_2}) through by $S(f-a)$ yields (\ref{Orbit5KKL})
\end{proof}
\begin{Proposition}\label{lastprop}  We have the Hamming type $444$ $(K,K,K)$ relation\begin{align}\label{KKK444} &   \Bb{1-a}{e-a}{f-a}{g-a}\Bb
{b}{1+b-e}{1+b-f}{1+b-g} \nr\times\,& B_5({e - a, b, 1 + b - f, g - a,  1 + b - a, 1 + b + c - f,1+b+d-f}) 
K_{p_0}(\vec{x}) 
 \nr+\,& \Bb{a}{1+a-e}{f-d}{g-d}\Bb
{b}{1+b-e}{f-c}{g-c}\nr\times\,&  B_5({1 + b - e, b, f - c, g - c,  1 + a + b - e, 1 + b - c,1+b+d-e}) 
 K_{n_4}(\vec{x}) 
\nr+\,&   \ds  \Bb{a}{b}{c}{d}\Bb
{e-a}{e-b}{e-c}{e-d}  B_5(\vec{x}) 
K_{p_5}(\vec{x}) =0. \end{align} \end{Proposition}
\begin{proof}Multiplying the result of Proposition 7.6 in \cite{FGS} through by $S(b)/\gg{b}$ (and recalling that the function $K(\vec{x})$ there is $\pi/\gg{1-a}$ times the one in use here), we get (in terms of the present definition of $K(\vec{x})$)
\begin{align}\label{KKK444_2} & \Bb{1-a}{e-a}   {f-a}{g-a}  \Bb{b}{1+b-e}{1+b-f}{1+b-g} \nr\times\,&\frac{S( a,b)}{S(f)} \ds \biggl[\genfrac{} {}{0pt}{}{\sin \pi (c+d-g) \sin \pi e \,\sin  \pi (f-c) \sin \pi (f-d) }{-\sin  \pi (e-a-b) \sin \pi (f-e) \sin \pi c \,\sin\pi d}\biggr]    \nr\times\,&K_{p_0}(\vec{x})
 \nr+\,&\Bb{a}{1+a-e}{f-d}{g-d}\Bb{b}{1+b-e}  {f-c} {g-c}  
\nr\times\,&
   \frac{S(b,e-a)}{S(f-b-c)}\ds\biggl[\genfrac{} {}{0pt}{} {
\sin\pi (a-b)\sin\pi (g-a-d) \sin\pi c\,\sin\pi(e-d)}{
- \sin\pi (a+c-f)\sin\pi e \,\sin\pi(f-b)\sin\pi(g-a) } \biggr] \nr\times\,&  K_{n_4}(\vec{x})\nr+\,&  \Bb{a} {b}{c}{d}  \Bb{e-a} {e-b}{e-c}{e-d}
\nr\times\,& \frac{S(b,e-b)}{S (c-b)}\ds\biggl[\genfrac{} {}{0pt}{} {  \sin\pi (a-c)\sin\pi (e-a-b)\sin\pi(f-b)\sin\pi(g-b)}{ -
\sin\pi (a-b)\sin\pi (e-a-c) \sin\pi(f-c)\sin\pi(g-c) } \biggr]  \nr\times\,&  
K_{p_5}(\vec{x}) =0.\end{align}Applying Lemma \ref{triglemma}(b) towards simplification of the terms in square brackets, in  (\ref{KKK444_2}), yields (\ref{KKK444}). \end{proof}


\begin{thebibliography}{99}

\bibitem{Ba1} W.N. Bailey, Transformations
of well-poised hypergeometric series, Proc. Lond. Math. Soc. {\bf 36}
(1934), no. 2, 235--240.

\bibitem{Ba} W.N. Bailey, Generalized Hypergeometric
Series, Cambridge University Press,
Cambridge, 1935.

\bibitem{Bar1} E.W. Barnes, A new development of the
theory of hypergeometric functions,
Proc. Lond. Math. Soc. {\bf 2} (1908), no. 6,
141--177.

\bibitem{Bar2} E.W. Barnes, A transformation of
generalized hypergeometric series,
Quart. J. Math. {\bf 41} (1910), 136--140.

\bibitem{BLS} W.A. Beyer, J.D. Louck, P.R. Stein,
Group theoretical basis of some identities for the
generalized hypergeometric series,
J. Math. Phys. {\bf 28} (1987), no. 3, 497--508.

\bibitem{BRS} F.J. van de Bult, E.M. Rains,
J.V. Stokman, Properties of generalized
univariate hypergeometric functions,
Comm. Math. Phys. {\bf 275} (2007),
no. 1, 37--95.

\bibitem{Bu} D. Bump, Barnes' second lemma and its
application to Rankin--Selberg convolutions,
Amer. J. Math. {\bf 110} (1988), 
no. 1, 179--185.

\bibitem{Dra} G. Drake (ed.), Springer Handbook of
Atomic, Molecular and Optical Physics,
Springer, New York, 2006.

\bibitem{FGS} M. Formichella, R.M. Green, E. Stade,
Coxeter group actions on ${}_4F_3(1)$
hypergeometric series,
Ramanujan J. {\bf 24} (2011), no. 1,
93--128.

\bibitem{GR} G. Gasper, M. Rahman,
Basic Hypergeometric Series,
Cambridge University Press, Cambridge, 1990.

\bibitem{Ga} C.F. Gauss, Disquisitiones
generales circa seriem infinitam
$1+\frac{\alpha\beta}{1 \times \gamma}x
+\frac{\alpha(\alpha+1)\beta(\beta+1)}
{1 \times 2 \times \gamma(\gamma+1)}xx
+\mbox{etc.}$, Werke 3,
K\"onigliche Gesellschaft der Wissenschaften,
G\"ottingen, 1876, pp. 123--162.

\bibitem{Gro} W. Groenevelt, The Wilson
function transform, Int. Math. Res. Not.
{\bf 2003,} no. 52, 2779--2817.

\bibitem{Groz} A. Grozin, Lectures on QED
and QCD: Practical Calculation and Renormalization
of One- and Multi-Loop Feynman Diagrams,
World Scientific, Singapore, 2007.

\bibitem{Har} G.H. Hardy, Ramanujan: Twelve
Lectures on Subjects Suggested by His Life and Work,
Cambridge University Press, Cambridge, 1940.

\bibitem{Hum} J.E. Humphreys, Reflection Groups and
Coxeter Groups, Cambridge University Press,
Cambridge, 1990.

\bibitem{KR} C. Krattenthaler, T. Rivoal,
How can we escape Thomae's relations?,
J. Math. Soc. Japan {\bf 58} (2006), no. 1,
183--210.

\bibitem{LJ1} S. Lievens, J. Van der Jeugt, Invariance
groups of three term transformations for basic
hypergeometric series, J. Comput.
Appl. Math. {\bf 197} (2006), no. 1, 1--14.

\bibitem{LJ2} S. Lievens, J. Van der Jeugt, Symmetry
groups of Bailey's transformations for
${}_{10}\phi_9$-series, J. Comput.
Appl. Math. {\bf 206} (2007), no. 1, 498--519.

\bibitem{M} I. Mishev, Coxeter group actions on
supplementary
pairs of Saalsch\"utzian ${}_4F_3(1)$
hypergeometric series, Ph.D. Thesis, University
of Colorado, 2009.

\bibitem{M1} I. Mishev,
Coxeter group actions on Saalsch\"utzian ${}_4F_3(1)$ series and
very-well-poised ${}_7F_6(1)$ series, J. Math. Anal. Appl. {\bf 385}
(2012), no. 2, 1119--1133.

\bibitem{R} J. Raynal, On the definition and
properties of generalized 6-$j$ symbols,
J. Math. Phys.
{\bf 20} (1979), no. 12, 2398--2415.

\bibitem{Roy} M. Roy, Coxeter Group Actions on Complementary Pairs of Very Well-Poised ${}_9 F_8(1)$ Hypergeometric Series,  Ph.D. Thesis, University of Colorado,
2011.

\bibitem{Rao} K. Srinivasa Rao, J. Van der Jeugt,
J. Raynal, R. Jagannathan, V. Rajeswari,
Group theoretical basis for the terminating
${}_3F_2(1)$ series, J.
Phys. A {\bf 25} (1992), no. 4, 861--876.

\bibitem{St1} E. Stade, Hypergeometric series and
Euler factors at infinity for $L$-functions on
$GL(3,\mathbb{R}) \times GL(3,\mathbb{R})$,
Amer. J. Math. {\bf 115} (1993), no. 2, 371--387.

\bibitem{St2} E. Stade, Mellin transforms of Whittaker
functions on $GL(4,\mathbb{R})$ and $GL(4,\mathbb{C})$,
Manuscripta Math. {\bf 87} (1995), no. 4, 511--526.

\bibitem{St3} E. Stade, Mellin transforms of
$GL(n,\mathbb{R})$ Whittaker functions,
Amer. J. Math. {\bf 123} (2001), no. 1, 121--161.

\bibitem{St4} E. Stade, Archimedean $L$-factors on
$GL(n) \times GL(n)$ and generalized Barnes integrals,
Israel J. Math. {\bf 127} (2002), 201--220.

\bibitem{ST} E. Stade, J. Taggart, Hypergeometric
series, a Barnes-type lemma, and Whittaker functions,
J. London Math. Soc. (2) {\bf 61} (2000), no. 1,
177--196.

\bibitem{T} J. Thomae, Ueber die Funktionen welche
durch Reihen der Form dargestellt werden:
$1+\frac{pp'p''}{1q'q''}+ \cdots$, J. Reine Angew.
Math. {\bf 87} (1879), 26--73.

\bibitem{V} J. Van der Jeugt, K. Srinivasa Rao,
Invariance groups of transformations of
basic hypergeometric series, J. Math. Phys.
{\bf 40} (1999), no. 12, 6692--6700.

\bibitem{Wh1} F.J.W. Whipple, A group of generalized
hypergeometric series: relations between 120 allied
series of the type $F(a,b,c;e,f)$,
Proc. London Math. Soc. {\bf 23} (1925), no. 2, 247--263.

\bibitem{Wh2} F.J.W. Whipple, Well-poised series and
other generalized hypergeometric series,
Proc. London Math. Soc. {\bf 25} (1926), no. 2, 525--544.

\bibitem{Wh3} F.J.W. Whipple, Relations between
well-poised hypergeometric series of the type
${}_7F_6$, Proc. London Math. Soc. {\bf 40} (1936),
no. 2, 336--344.

\end{thebibliography}
\end{document}